\documentclass[11pt]{amsart} \textwidth=14.5cm \oddsidemargin=1cm
\usepackage{amsmath}
\usepackage{amsxtra}
\usepackage{amscd}
\usepackage{amsthm}
\usepackage{amsfonts}
\usepackage{amssymb}
\usepackage{eucal}
\usepackage{pmgraph}
\usepackage{stmaryrd}
\usepackage[usenames,dvipsnames]{color}

\input prepictex
\input pictex
\input postpictex

\newtheorem{thm}{Theorem}[section]
\newtheorem{cor}[thm]{Corollary}
\newtheorem{lem}[thm]{Lemma}
\newtheorem{prop}[thm]{Proposition}

\theoremstyle{definition}

\newtheorem{example}[thm]{Example}

\theoremstyle{remark}
\newtheorem{rem}[thm]{Remark}

\numberwithin{equation}{section}

\begin{document}

\newcommand{\thmref}[1]{Theorem~\ref{#1}}
\newcommand{\secref}[1]{Section~\ref{#1}}
\newcommand{\lemref}[1]{Lemma~\ref{#1}}
\newcommand{\propref}[1]{Proposition~\ref{#1}}
\newcommand{\corref}[1]{Corollary~\ref{#1}}
\newcommand{\remref}[1]{Remark~\ref{#1}}
\newcommand{\eqnref}[1]{(\ref{#1})}
\newcommand{\exref}[1]{Example~\ref{#1}}

\newcommand{\nc}{\newcommand}
\nc{\Z}{{\mathbb Z}}
\nc{\C}{{\mathbb C}}
\nc{\N}{{\mathbb N}}
\nc{\F}{{\mf F}}
\nc{\Q}{\ol{Q}}
\nc{\la}{\lambda}
\nc{\ep}{\epsilon}
\nc{\h}{\mathfrak h}
\nc{\n}{\mf n}
\nc{\G}{{\mathfrak g}}
\nc{\DG}{\widetilde{\mathfrak g}}
\nc{\SG}{\overline{\mathfrak g}}
\nc{\D}{\mc D} \nc{\Li}{{\mc L}} \nc{\La}{\Lambda} \nc{\is}{{\mathbf
i}} \nc{\V}{\mf V} \nc{\bi}{\bibitem} \nc{\NS}{\mf N}
\nc{\dt}{\mathord{\hbox{${\frac{d}{d t}}$}}} \nc{\E}{\mc E}
\nc{\ba}{\tilde{\pa}} \nc{\half}{\frac{1}{2}} \nc{\mc}{\mathcal}
\nc{\mf}{\mathfrak} \nc{\hf}{\frac{1}{2}}
\nc{\hgl}{\widehat{\mathfrak{gl}}} \nc{\gl}{{\mathfrak{gl}}}
\nc{\hz}{\hf+\Z}
\nc{\dinfty}{{\infty\vert\infty}} \nc{\SLa}{\overline{\Lambda}}
\nc{\SF}{\overline{\mathfrak F}} \nc{\SP}{\overline{\mathcal P}}
\nc{\U}{\mathfrak u} \nc{\SU}{\overline{\mathfrak u}}
\nc{\ov}{\overline}
\nc{\wt}{\widetilde}
\nc{\osp}{\mf{osp}}
\nc{\spo}{\mf{spo}}
\nc{\hosp}{\widehat{\mf{osp}}}
\nc{\hspo}{\widehat{\mf{spo}}}
\nc{\I}{\mathbb{I}}
\nc{\X}{\mathbb{X}}
\nc{\Y}{\mathbb{Y}}
\nc{\hh}{\widehat{\mf{h}}}
\nc{\cc}{{\mathfrak c}}
\nc{\dd}{{\mathfrak d}}
\nc{\aaa}{{\mf A}}
\nc{\xx}{{\mf x}}
\nc{\wty}{\widetilde{\mathbb Y}}
\nc{\ovy}{\overline{\mathbb Y}}
\nc{\vep}{\bar{\epsilon}}

\advance\headheight by 2pt

\title[Super duality and ortho-symplectic Lie superalgebras]
{Super duality and irreducible characters of ortho-symplectic Lie
superalgebras}

\author[Cheng]{Shun-Jen Cheng}
\address{Institute of Mathematics, Academia Sinica, Taipei,
Taiwan 10617} \email{chengsj@math.sinica.edu.tw}

\author[Lam]{Ngau Lam}
\address{Department of Mathematics, National Cheng-Kung University, Tainan, Taiwan 70101}
\email{nlam@mail.ncku.edu.tw}

\author[Wang]{Weiqiang Wang}
\address{Department of Mathematics, University of Virginia, Charlottesville, VA 22904}
\email{ww9c@virginia.edu}

\begin{abstract}
We formulate and establish a super duality which connects parabolic categories $O$ for
the ortho-symplectic Lie superalgebras and classical Lie algebras of $BCD$ types. This
provides a complete and conceptual solution of the irreducible character problem for the
ortho-symplectic Lie superalgebras in a parabolic category $O$, which includes all finite
dimensional irreducible modules, in terms of classical Kazhdan-Lusztig polynomials.
\end{abstract}



 \maketitle

  \setcounter{tocdepth}{1}
\tableofcontents

\section{Introduction}

\subsection{}
\label{sec:LA}

Finding the irreducible characters is a fundamental problem in
representation theory. As a prototype of this problem, consider a
complex semisimple Lie algebra $\G$. The problem is solved in two
steps, following the historical development:
\begin{enumerate}
\item The category of finite dimensional $\G$-modules is semisimple and the
    corresponding irreducible $\G$-characters are given by Weyl's character formula.

\item A general solution to the irreducible character problem in
the BGG category $O$ was given much later by the Kazhdan-Lusztig
(KL) polynomials (theorems of Beilinson-Bernstein and
Brylinski-Kashiwara) \cite{KL, BB, BK}.
\end{enumerate}

%
%
\subsection{}
\label{sec:LSA}

The study of Lie superalgebras and their representations was largely motivated by the
notion of supersymmetry in physics. A Killing-Cartan type classification of finite
dimensional complex simple Lie superalgebras was achieved by Kac \cite{K1} in 1977. The
most important subclass of the simple Lie superalgebras (called basic classical),
including two infinite series of types $A$ and $\osp$, bears a close resemblance to the
usual simple Lie algebras, so we can make sense of root systems, Dynkin diagrams,
triangular decomposition, Cartan and Borel subalgebras, (parabolic) category $O$, and so
on.

However, the representation theory of Lie superalgebras $\SG$ has encountered several
substantial difficulties, as made clear by numerous works over the last three decades
(cf. \cite{BL, K2, KW, LSS, Pen, dHKT, dJ} as a sample of earlier literature, and some
more recent references which can be found in the next paragraph):
\begin{enumerate}
\item There exist non-conjugate Borel subalgebras for a given Lie
superalgebra $\SG$.

\item The category $\mc F$ of finite dimensional $\SG$-modules is in general not
    semisimple. A uniform Weyl type finite dimensional character formula does not
    exist.

\item One has a notion of a Weyl group associated to the even subalgebra of $\SG$;
    however, the linkage in the category $O$ (or in $\mc F$) of $\SG$-modules is not
    solely controlled by the Weyl group.

\item A block in the category $O$ (or in $\mc F$) may contain
infinitely many simple objects.
%
\end{enumerate}
The conventional wisdom of solving the irreducible character
problem for Lie superalgebras has been to follow closely the two
steps for Lie algebras in \ref{sec:LA}. As the problem is already
very difficult in the category $\mc F$, there has been little
attempt in understanding the category $O$.

%
%
For type $A$ Lie superalgebra $\gl(m|n)$, there have  been several different general
approaches over the years. Serganova \cite{Se} in 1996 developed a mixed geometric and
algebraic approach to solving the irreducible character problem in the category $\mc F$.
Brundan in 2003 \cite{B} developed a new elegant purely algebraic solution to the same
problem in $\mc F$ using Lusztig-Kashiwara canonical basis. Developing the idea of super
duality \cite{CW2} (which generalizes \cite{CWZ}) which connects the categories $O$ for
Lie superalgebras and Lie algebras of type $A$ for the first time, two of the authors
\cite{CL} very recently established the super duality conjecture therein. In particular
they provided a complete solution to the irreducible character problem for a fairly
general parabolic category $O$ (including $\mc F$ as a very special case) in terms of KL
polynomials of  type $A$. Independently, Brundan and Stroppel \cite{BS} proved the super
duality conjecture  in \cite{CWZ}, offering yet another solution of the irreducible
character problem in $\mc F$.

\subsection{}
The goal of this paper is to formulate and establish a super duality which connects
parabolic category $O$ for Lie superalgebra of type $\osp$ with parabolic category $O$
for classical Lie algebras of types $BCD$, vastly generalizing the type $A$ case of \cite{CWZ,
CW2, CL}. In particular, it provides a complete solution of the irreducible character
problem for $\osp$ in some suitable parabolic category $O$, which includes all finite
dimensional irreducibles, in terms of parabolic KL polynomials of $BCD$ types
(cf.~Deodhar \cite{Deo}).

\subsection{}

Before launching on a detailed explanation of our main ideas below, it is helpful to keep
in mind the analogy that the ring of symmetric functions (or its super counterpart) in
infinitely many variables carries more symmetries than in finitely many variables, and a
truncation process can easily recover finitely many variables. The super duality can be
morally thought as a categorification of the standard involution $\omega$ on the ring of
symmetric functions and it only becomes manifest when the underlying Lie (super)algebras
pass to infinite rank. Then truncation functors can be used to recover the finite rank
cases which we are originally interested in.

Even though the finite dimensional Lie superalgebras of type $\osp$ depend on two
integers $m$ and $n$, our view is to fix one and let the other, say $n$, vary, and so let
us denote an $\osp$ Lie superalgebra by $\SG_n$. By choosing appropriately a Borel and a
Levi subalgebra of $\SG_n$, we formulate a suitable parabolic category $\ov{\mc{O}}_n$ of
$\SG_n$-modules. It turns out that there are four natural choices one can make here which
correspond to the four Dynkin diagrams $\mf{b,b^\bullet,c,d}$ in \ref{dynkin} (the type
$\mf a$ case has been treated in \cite{CL}). There is a natural sequence of inclusions of
Lie superalgebras:
$$
\SG_1\subset \SG_2 \subset \ldots \subset \SG_n \subset \ldots.
$$
Let $\SG :=\cup_{n=1}^\infty \SG_n$. A suitable category $\ov{\mc{O}}$ of $\SG$-modules
can be identified with the inverse limit $\lim\limits_{\longleftarrow} \ov{\mc{O}}_n$. On
the other hand, we introduce truncation functors
$\mf{tr}^{k}_n:\ov{\mc{O}}_k\rightarrow\ov{\mc{O}}_n$ for $k>n$, as analogues of the
truncation functors studied in algebraic group setting ( cf. Donkin \cite{Don}). These
truncation functors send parabolic Verma modules to parabolic Vermas or zero and
irreducibles to irreducibles or zero. In particular, this allows us to derive the
irreducible characters in $\ov{\mc{O}}_n$ once we know those in $\ov{\mc O}$.

Corresponding to each of the above choices of $\SG_n$ and $\ov{\mc O}_n$, we have the Lie
algebra counterparts $\G_n$ and parabolic categories ${\mc{O}}_n$ for positive integers
$n$. Moreover, we have natural inclusions of Lie algebras $\G_n \subset \G_{n+1}$, for
all $n$, which allow us to define the Lie algebra $\G :=\cup_n \G_n$ and the parabolic
category ${\mc{O}}$ of $\G$-modules. Similarly, the category ${\mc{O}}$ can be identified
with the inverse limit $\lim\limits_{\longleftarrow} {\mc{O}}_n$. In the main body of the
paper, we actually replace $\G, \SG$ et cetera by their (trivial) central extensions. The
reason is that the truncation functors depend implicitly on a stabilization scalar, which
is interpreted conceptually as a level of representations with respect to the central
extensions.

To establish a connection between $\ov{\mc{O}}$ and ${\mc{O}}$, we introduce another
infinite rank Lie superalgebra $\wt \G$ and its parabolic category $\wt{\mc{O}}$. The Lie
superalgebra $\wt \G$ contains $\G$ and $\SG$ as natural subalgebras (though not as Levi
subalgebras), and this enables us to introduce two natural functors $T: \wt{\mc{O}}
\rightarrow \mc O$ and  $\ov{T}: \wt{\mc{O}} \rightarrow \ov{\mc O}$. Using the technique
of odd reflections among others, we establish in \secref{sec:character} a key property
that $T$ and $\ov T$ respect the parabolic Verma and irreducible modules, respectively.
This result is already sufficient to provide a complete solution of irreducible
$\osp$-characters in the category $\ov{\mc{O}}_n$ in the first half of the paper (by the
end of Section~\ref{sec:T}). We remark that the idea of introducing an auxiliary Lie
superalgebra $\wt \G$ and category $\wt{\mc O}$ has been used in the type $A$
superalgebra setting \cite{CL}.

Recall that for the usual category $O$ of Lie algebras, the KL polynomials were
interpreted by Vogan \cite{V} in terms of Kostant $\mf u$-homology groups. The $\mf
u$-homology groups make perfect sense for Lie superalgebras, and we may take this
interpretation as the definition for the otherwise undefined KL polynomials in category
$O$ for Lie superalgebras (cf.~Section \ref{sec:LSA}~(3)), as Serganova \cite{Se} did in
the category $\mc F$ of $\gl (m|n)$-modules. In Section~\ref{sec:homology}, we show that
the functors $T$ and $\ov T$ match the corresponding $\mf u$-homology groups and hence
the corresponding KL polynomials (compare \cite{CKW}). Actually the computation in
\cite{CKW} of the $\mf u$-homology groups with coefficients in the Lie superalgebra
oscillator modules via Howe duality was the first direct supporting evidence for the
super duality for $\osp$ as formulated in this paper.

Section~\ref{sec:category} of the paper is devoted to proving that both $T$ and $\ov T$
are indeed category equivalences. As a consequence, we have established that the
categories $\mc O$ and $\ov{\mc O}$ are equivalent, which is called super duality. A
technical difference here from \cite{CL} is that we need to deal with the fact that
parabolic Verma modules in $\ov{\mc O}$ may not have finite composition series. An
immediate corollary of the super duality is that any known BGG resolution in the category
$\mc O$ gives rise to a BGG type resolution in the category $\ov{\mc O}$, and vice versa.

\subsection{}
The finite dimensional irreducible $\osp$-modules are of highest weight, and they are
classified in terms of the Dynkin labels by Kac \cite{K2}. We note that the finite
dimensional irreducible modules of non-integral highest weights are typical and so their
characters are known \cite[Theorem 1]{K2}. It turns out that a more natural labeling of
the remaining finite dimensional irreducible $\osp$-modules (of integral highest weights)
is given in terms of Young diagrams just as for classical Lie algebras (see e.g.
\cite{SW} for such a formulation and a new proof using odd reflections).

As Borel subalgebras are not conjugate to each other, it becomes a nontrivial problem to
find the extremal weights, i.e.~highest weights with respect to different Borel, of a
given finite dimensional irreducible $\osp$-module. We provide an elegant and simple
answer in terms of a combinatorial notion which we call {\em block Frobenius coordinates}
associated to Young diagrams.

We observe that our solution of the irreducible character problem in $\ov{\mc O}_n$
includes solutions to all finite dimensional irreducible $\osp$-characters.

The category $\mc F$ for a general $\osp$ (with the exception of $\osp(2|2n)$) is not a
highest weight category and does not admit an abstract KL theory in the sense of Cline,
Parshall and Scott \cite{CPS}, as indicated in the case of $\osp(3|2)$
\cite[Section~2]{Ger}. For a completely independent and different approach to the finite
dimensional irreducible $\osp$-characters in the category $\mc F$, see Gruson and
Serganova \cite{GS}. The finite dimensional irreducible characters of $\osp(2|2n)$ were
obtained in \cite{dJ}. The finite dimensional irreducible characters of $\osp(k|2)$ were
also computed in \cite{SZ}.

\subsection{}

In hindsight, here is how our super duality approach overcomes the
difficulties as listed in \ref{sec:LSA}.
\begin{enumerate}
\item The existence of non-conjugate Borel subalgebras for a Lie
superalgebra is essential for establishing the properties of the
functors $T$ and $\ov T$. Choices of suitable Borel subalgebras
are crucial for a formulation of the compatible sequence of
categories $\ov{\mc O}_n$ for $n>0$.

\item The category $\mc F$ of finite dimensional $\SG_n$-modules does not play any
    special role in our approach. Even the ``natural" $\osp(M|2n)$-modules
    $\C^{M|2n}$ do not correspond well with each other under truncation functors, as
    they are natural with respect to the ``wrong" Borel.

\item In the $n \to \infty$ limit, the linkage in the category
$\ov{\mc O}$ of $\SG$-modules is completely controlled by the Weyl
group of the corresponding Lie algebra $\G$ (which contains the
even subalgebra of $\SG$ as a subalgebra).

\item In the $n \to \infty$ limit, it is no surprise for a block
to contain infinitely many simple objects.
%
\end{enumerate}

In the extreme cases described in \secref{sec:m=0}, we indeed obtain an equivalence of
module categories between two classical (non-super!) Lie algebras of types $C$ and $D$ of
infinite rank at opposite levels. If one is willing to regard $\osp(1|\infty)$ as
classical (recall that the finite dimensional $\osp(1|2n)$-module category is
semisimple), there is another equivalence of categories which relates $\osp(1|\infty)$ to
the infinite rank Lie algebra of type $B$. In this sense, our super duality has a flavor
of the Langlands duality.

The super duality approach here can be further adapted to the
setting of Kac-Moody superalgebras (including affine
superalgebras) and this will shed new light on the irreducible
character problem for these superalgebras. The details will appear
elsewhere.

It is well known that the proof of KL conjectures involves deep geometric machinery and
results on $D$-modules of flag manifolds. The formulation of super duality suggests
potential direct connections on the (super) geometric level behind the categories $\mc O$
and $\ov{\mc O}$, which will be very important to develop.

\subsection{}

The paper is organized as follows. In \secref{sec:superalgebras} the Lie superalgebras
$\G$, $\SG$ and $\DG$ are defined, with their respective module categories ${\mc{O}}$,
$\ov{\mc{O}}$ and $\wt{\mc{O}}$ introduced in \secref{sec:O}. In \secref{sec:character},
we provide a complete solution of the irreducible $\osp$ character problem in category
$\ov{\mc O}_n$ for all $n$, including all finite dimensional irreducible
$\osp$-characters, in terms of the KL polynomials of classical type.  We establish in
\secref{sec:category} equivalence of the categories $\mc O$ and $\ov{\mc O}$.
\secref{finite:dim:repn} offers a diagrammatic description of the extremal weights of the
finite dimensional irreducible $\osp$-modules.

Throughout the paper the symbols $\Z$, $\N$, and $\Z_+$ stand for the sets of all,
positive and non-negative integers, respectively. All vector spaces, algebras, tensor
products, et cetera, are over the field of complex numbers $\C$.


{\bf Acknowledgment.} The first author is partially supported by
an NSC-grant and an Academia Sinica Investigator grant, and he
thanks NCTS/TPE and the Department of Mathematics of University of
Virginia for hospitality and support. The second author is
partially supported by an NSC-grant and
thanks NCTS/SOUTH. The third author is partially
supported by NSF and NSA grants, and he thanks the Institute of
Mathematics of Academia Sinica in Taiwan for hospitality and
support. The results of the paper were announced by the first
author in the AMS meeting at Raleigh in April 2009, and they were
presented by the third author in conferences at Ottawa, Canada and
Durham, UK in July 2009.

\section{Lie superalgebras of infinite rank}\label{sec:superalgebras}

In this section, we introduce infinite rank Lie (super)algebras $\G^\xx$, $\SG^\xx$ and
$\wt{\G}^\xx$ associated to the $3$ Dynkin diagrams in \eqref{Dynkin:combined} below,
where $\xx$ denotes one of the four types $\mf{b,b^\bullet,c,d}$.

\subsection{Dynkin diagrams of $\G^\xx$, $\ov{\G}^\xx$ and $\wt{\G}^\xx$}
\label{dynkin}

Let $m\in\Z_+$. Consider the free abelian group with basis
$\{\epsilon_{-m},\ldots,\epsilon_{-1}\}\cup\{\epsilon_{r}\vert
r\in\hf\N\}$, with a symmetric bilinear form $(\cdot|\cdot)$
given by
\begin{align*}
(\epsilon_r|\epsilon_s)=(-1)^{2r}\delta_{rs}, \qquad r,s \in
\{-m,\ldots,-1\} \cup \hf \N.
\end{align*}
We set
\begin{align}\label{alpha:beta}
&\alpha_{\times}:=\epsilon_{-1}-\epsilon_{1/2},
 \quad\alpha_{j}
:=\epsilon_{j}-\epsilon_{j+1},\quad -m\le j\le -2,\\
&\beta_{\times}:=\epsilon_{-1}-\epsilon_{1},\quad
\alpha_{r}:=\epsilon_{r}-\epsilon_{r+1/2},\quad \beta_{r}
:=\epsilon_{r}-\epsilon_{r+1},\quad r\in\hf\N.\nonumber
\end{align}

For $\xx =\mf{b,b^\bullet,c,d}$, we denote by $\mf{k}^\xx$ the contragredient Lie
(super)algebras (\cite[Section 2.5]{K1}) whose Dynkin diagrams
\makebox(23,0){$\oval(20,12)$}\makebox(-20,8){$\mf{k}^\xx$} together with certain
distinguished sets of simple roots $\Pi(\mf{k^x})$ are listed as follows: \vspace{.3cm}

\begin{center}
\hskip -3cm \setlength{\unitlength}{0.16in}
\begin{picture}(24,2)
\put(8,2){\makebox(0,0)[c]{$\bigcirc$}}
\put(10.4,2){\makebox(0,0)[c]{$\bigcirc$}}
\put(14.85,2){\makebox(0,0)[c]{$\bigcirc$}}
\put(17.25,2){\makebox(0,0)[c]{$\bigcirc$}}
\put(19.4,2){\makebox(0,0)[c]{$\bigcirc$}}
\put(5.6,2){\makebox(0,0)[c]{$\bigcirc$}}
\put(8.4,2){\line(1,0){1.55}} \put(10.82,2){\line(1,0){0.8}}
\put(13.2,2){\line(1,0){1.2}} \put(15.28,2){\line(1,0){1.45}}
\put(17.7,2){\line(1,0){1.25}}
\put(6,1.8){$\Longleftarrow$}
\put(12.5,1.95){\makebox(0,0)[c]{$\cdots$}}
\put(-.5,2){\makebox(0,0)[c]{$\mf{b}$:}}
\put(5.5,1){\makebox(0,0)[c]{\tiny$-\epsilon_{-m}$}}
\put(8,1){\makebox(0,0)[c]{\tiny$\alpha_{-m}$}}
\put(17.2,1){\makebox(0,0)[c]{\tiny$\alpha_{-3}$}}
\put(19.3,1){\makebox(0,0)[c]{\tiny$\alpha_{-2}$}}
\end{picture}
\end{center}
\begin{center}
\hskip -3cm \setlength{\unitlength}{0.16in}
\begin{picture}(24,2)
\put(5.6,2){\circle*{0.9}} \put(8,2){\makebox(0,0)[c]{$\bigcirc$}}
\put(10.4,2){\makebox(0,0)[c]{$\bigcirc$}}
\put(14.85,2){\makebox(0,0)[c]{$\bigcirc$}}
\put(17.25,2){\makebox(0,0)[c]{$\bigcirc$}}
\put(19.4,2){\makebox(0,0)[c]{$\bigcirc$}}
\put(8.35,2){\line(1,0){1.5}} \put(10.82,2){\line(1,0){0.8}}
\put(13.2,2){\line(1,0){1.2}} \put(15.28,2){\line(1,0){1.45}}
\put(17.7,2){\line(1,0){1.25}}
\put(6.8,2){\makebox(0,0)[c]{$\Longleftarrow$}}
\put(12.5,1.95){\makebox(0,0)[c]{$\cdots$}}
\put(-.2,2){\makebox(0,0)[c]{$\mf{b^\bullet}$:}}
\put(5.5,1){\makebox(0,0)[c]{\tiny$-\epsilon_{-m}$}}
\put(8,1){\makebox(0,0)[c]{\tiny$\alpha_{-m}$}}
\put(17.2,1){\makebox(0,0)[c]{\tiny$\alpha_{-3}$}}
\put(19.3,1){\makebox(0,0)[c]{\tiny$\alpha_{-2}$}}
\end{picture}
\end{center}
\begin{center}
\hskip -3cm \setlength{\unitlength}{0.16in}
\begin{picture}(24,2)
\put(5.7,2){\makebox(0,0)[c]{$\bigcirc$}}
\put(8,2){\makebox(0,0)[c]{$\bigcirc$}}
\put(10.4,2){\makebox(0,0)[c]{$\bigcirc$}}
\put(14.85,2){\makebox(0,0)[c]{$\bigcirc$}}
\put(17.25,2){\makebox(0,0)[c]{$\bigcirc$}}
\put(19.4,2){\makebox(0,0)[c]{$\bigcirc$}}
\put(6.8,2){\makebox(0,0)[c]{$\Longrightarrow$}}
\put(8.4,2){\line(1,0){1.55}} \put(10.82,2){\line(1,0){0.8}}
\put(13.2,2){\line(1,0){1.2}} \put(15.28,2){\line(1,0){1.45}}
\put(17.7,2){\line(1,0){1.25}}
\put(12.5,1.95){\makebox(0,0)[c]{$\cdots$}}
\put(-.5,2){\makebox(0,0)[c]{$\mf{c}$:}}
\put(5.5,1){\makebox(0,0)[c]{\tiny$-2\epsilon_{-m}$}}
\put(8,1){\makebox(0,0)[c]{\tiny$\alpha_{-m}$}}
\put(17.2,1){\makebox(0,0)[c]{\tiny$\alpha_{-3}$}}
\put(19.3,1){\makebox(0,0)[c]{\tiny$\alpha_{-2}$}}
\end{picture}
\end{center}
\begin{center}
\hskip -3cm \setlength{\unitlength}{0.16in}
\begin{picture}(24,3.5)
\put(8,2){\makebox(0,0)[c]{$\bigcirc$}}
\put(10.4,2){\makebox(0,0)[c]{$\bigcirc$}}
\put(14.85,2){\makebox(0,0)[c]{$\bigcirc$}}
\put(17.25,2){\makebox(0,0)[c]{$\bigcirc$}}
\put(19.4,2){\makebox(0,0)[c]{$\bigcirc$}}
\put(6,3.8){\makebox(0,0)[c]{$\bigcirc$}}
\put(6,.3){\makebox(0,0)[c]{$\bigcirc$}}
\put(8.4,2){\line(1,0){1.55}} \put(10.82,2){\line(1,0){0.8}}
\put(13.2,2){\line(1,0){1.2}} \put(15.28,2){\line(1,0){1.45}}
\put(17.7,2){\line(1,0){1.25}}
\put(7.6,2.2){\line(-1,1){1.3}}
\put(7.6,1.8){\line(-1,-1){1.3}}
\put(12.5,1.95){\makebox(0,0)[c]{$\cdots$}}
\put(-.5,2){\makebox(0,0)[c]{$\mf{d}$:}}
\put(3.3,0.3){\makebox(0,0)[c]{\tiny${-}\epsilon_{-m}{-}\epsilon_{-m+1}$}}
\put(4.7,3.8){\makebox(0,0)[c]{\tiny$\alpha_{-m}$}}
\put(8.2,1){\makebox(0,0)[c]{\tiny$\alpha_{-m+1}$}}
\put(17.2,1){\makebox(0,0)[c]{\tiny$\alpha_{-3}$}}
\put(19.3,1){\makebox(0,0)[c]{\tiny$\alpha_{-2}$}}
\end{picture}
\end{center}
According to \cite[Proposition 2.5.6]{K1} these Lie (super)algebras are $\mf{so}(2m+1)$,
$\mf{osp}(1|2m)$, $\mf{sp}(2m)$ for $m\ge 1$ and $\mf{so}(2m)$ for $m \ge 2$,
respectively. We will use the same notation
\makebox(23,0){$\oval(20,12)$}\makebox(-20,8){$\mf{k}^\xx$} to denote the diagrams of all
the degenerate cases for $m=0,1$ as well.  (See Sections~\ref{sec:realize} and
\ref{sec:m=0} below). We have used \makebox(15,5){\circle*{7}} to denote an odd
non-isotropic simple root. So $\mf{osp}(1|2m)$ is actually a Lie superalgebra (instead of
Lie algebra), but it is classical from the super duality viewpoint in this paper.

For $n\in\N$ let
\makebox(23,0){$\oval(20,12)$}\makebox(-20,8){$\mf{T}_n$},
\makebox(23,0){$\oval(20,14)$}\makebox(-20,8){$\ov{\mf{T}}_n$} and
\makebox(23,0){$\oval(20,14)$}\makebox(-20,8){$\wt{\mf{T}}_n$}
denote the following Dynkin diagrams, where $\bigotimes$ denotes
an odd isotropic simple root:
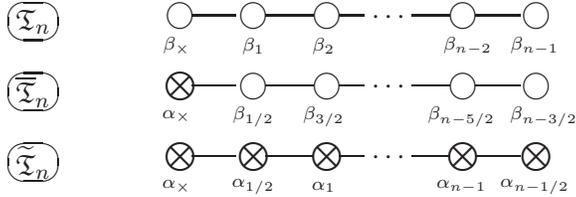
\begin{center}
\hskip -3cm \setlength{\unitlength}{0.16in}
\begin{picture}(24,3)
\put(8,2){\makebox(0,0)[c]{$\bigcirc$}}
\put(10.4,2){\makebox(0,0)[c]{$\bigcirc$}}
\put(14.85,2){\makebox(0,0)[c]{$\bigcirc$}}
\put(17.25,2){\makebox(0,0)[c]{$\bigcirc$}}
\put(5.6,2){\makebox(0,0)[c]{$\bigcirc$}}
\put(8.4,2){\line(1,0){1.55}} \put(10.82,2){\line(1,0){0.8}}
\put(13.2,2){\line(1,0){1.2}} \put(15.28,2){\line(1,0){1.45}}
\put(6,2){\line(1,0){1.4}}
\put(12.5,1.95){\makebox(0,0)[c]{$\cdots$}}
\put(0,1.2){{\ovalBox(1.6,1.2){$\mf{T}_n$}}}
\put(5.5,1){\makebox(0,0)[c]{\tiny$\beta_{\times}$}}
\put(8,1){\makebox(0,0)[c]{\tiny$\beta_{1}$}}
\put(10.3,1){\makebox(0,0)[c]{\tiny$\beta_{2}$}}
\put(15,1){\makebox(0,0)[c]{\tiny$\beta_{n-2}$}}
\put(17.2,1){\makebox(0,0)[c]{\tiny$\beta_{n-1}$}}
\end{picture}
\end{center}
\begin{center}
\hskip -3cm \setlength{\unitlength}{0.16in}
\begin{picture}(24,2)
\put(8,2){\makebox(0,0)[c]{$\bigcirc$}}
\put(10.4,2){\makebox(0,0)[c]{$\bigcirc$}}
\put(14.85,2){\makebox(0,0)[c]{$\bigcirc$}}
\put(17.25,2){\makebox(0,0)[c]{$\bigcirc$}}
\put(5.6,2){\makebox(0,0)[c]{$\bigotimes$}}
\put(8.4,2){\line(1,0){1.55}} \put(10.82,2){\line(1,0){0.8}}
\put(13.2,2){\line(1,0){1.2}} \put(15.28,2){\line(1,0){1.45}}
\put(6,2){\line(1,0){1.4}}
\put(12.5,1.95){\makebox(0,0)[c]{$\cdots$}}
\put(0,1.2){{\ovalBox(1.6,1.2){$\ov{\mf{T}}_n$}}}
\put(5.5,1){\makebox(0,0)[c]{\tiny$\alpha_{\times}$}}
\put(10.3,1){\makebox(0,0)[c]{\tiny$\beta_{3/2}$}}
\put(8,1){\makebox(0,0)[c]{\tiny$\beta_{1/2}$}}
\put(14.8,1){\makebox(0,0)[c]{\tiny$\beta_{n-5/2}$}}
\put(17.5,1){\makebox(0,0)[c]{\tiny$\beta_{n-3/2}$}}
\end{picture}
\end{center}
\begin{center}
\hskip -3cm \setlength{\unitlength}{0.16in}
\begin{picture}(24,2)
\put(8,2){\makebox(0,0)[c]{$\bigotimes$}}
\put(10.4,2){\makebox(0,0)[c]{$\bigotimes$}}
\put(14.85,2){\makebox(0,0)[c]{$\bigotimes$}}
\put(17.25,2){\makebox(0,0)[c]{$\bigotimes$}}
\put(5.6,2){\makebox(0,0)[c]{$\bigotimes$}}
\put(8.4,2){\line(1,0){1.55}} \put(10.82,2){\line(1,0){0.8}}
\put(13.2,2){\line(1,0){1.2}} \put(15.28,2){\line(1,0){1.45}}
\put(6,2){\line(1,0){1.4}}
\put(12.5,1.95){\makebox(0,0)[c]{$\cdots$}}
\put(0,1.2){{\ovalBox(1.6,1.2){$\wt{\mf{T}}_n$}}}
\put(5.5,1){\makebox(0,0)[c]{\tiny$\alpha_{\times}$}}
\put(8,1){\makebox(0,0)[c]{\tiny$\alpha_{1/2}$}}
\put(14.8,1){\makebox(0,0)[c]{\tiny$\alpha_{n-1}$}}
\put(17.2,1){\makebox(0,0)[c]{\tiny$\alpha_{n-1/2}$}}
\put(10.3,1){\makebox(0,0)[c]{\tiny$\alpha_{1}$}}
\end{picture}
\end{center}
The Lie superalgebras associated with these Dynkin diagrams are
$\gl(n+1)$, $\mf{gl}(1|n)$ and $\mf{gl}(n|n+1)$, respectively. In the
limit $n\to\infty$, the associated Lie superalgebras are direct limits of
these Lie superalgebras, and we
will simply drop $\infty$ to write
\makebox(23,0){$\oval(20,12)$}\makebox(-20,8){$\mf{T}$} =
\makebox(23,0){$\oval(20,12)$}\makebox(-20,8){$\mf{T}_\infty$} and
so on.

Any of the {\em head} diagrams
\makebox(23,0){$\oval(20,11)$}\makebox(-20,8){$\mf{k}^\xx$} may be
connected with the {\em tail} diagrams
\makebox(23,0){$\oval(20,12)$}\makebox(-20,8){$\mf{T}_n$},
\makebox(23,0){$\oval(20,14)$}\makebox(-20,8){$\ov{\mf{T}}_n$} and
\makebox(23,0){$\oval(20,14)$}\makebox(-20,8){$\wt{\mf{T}}_n$} to
produce the following Dynkin diagrams ($n\in\N\cup\{\infty\}$):
\begin{equation}\label{Dynkin:combined}
\hskip -3cm \setlength{\unitlength}{0.16in}
\begin{picture}(24,1)
\put(5.0,0.5){\makebox(0,0)[c]{{\ovalBox(1.6,1.2){$\mf{k}^\xx$}}}}
\put(5.8,0.5){\line(1,0){1.85}}
\put(8.5,0.5){\makebox(0,0)[c]{{\ovalBox(1.6,1.2){$\mf{T}_n$}}}}
\put(15,0.5){\makebox(0,0)[c]{{\ovalBox(1.6,1.2){$\mf{k}^\xx$}}}}
\put(15.8,0.5){\line(1,0){1.85}}
\put(18.5,0.5){\makebox(0,0)[c]{{\ovalBox(1.6,1.2){$\ov{\mf{T}}_n$}}}}
\put(25,0.5){\makebox(0,0)[c]{{\ovalBox(1.6,1.2){$\mf{k}^\xx$}}}}
\put(25.8,0.5){\line(1,0){1.85}}
\put(28.5,0.5){\makebox(0,0)[c]{{\ovalBox(1.6,1.2){$\wt{\mf{T}}_n$}}}}
\end{picture}
\end{equation}
We will denote the sets of simple roots of the above diagrams
accordingly by $\Pi_n^{\xx}$, $\ov{\Pi}_n^\xx$ and
$\wt{\Pi}_n^\xx$. For $n=\infty$, we also denote the sets of
positive roots by $\Phi_+^\xx$, $\ov{\Phi}^\xx_+$ and
$\wt{\Phi}^\xx_+$, and the sets of roots by $\Phi^\xx$,
$\ov{\Phi}^\xx$ and $\wt{\Phi}^\xx$, respectively.

\subsection{Realization} \label{sec:realize}

Let us denote the $3$ Dynkin diagrams of \eqnref{Dynkin:combined}
at $n=\infty$ by
\makebox(23,0){$\oval(20,14)$}\makebox(-20,8){$\G^\xx$},
\makebox(23,0){$\oval(20,14)$}\makebox(-20,8){$\SG^\xx$} and
\makebox(23,0){$\oval(20,14)$}\makebox(-20,8){$\DG^\xx$}. We
provide a realization for the corresponding Lie superalgebras.

 For $m\in\Z_+$ consider the following totally ordered set
$\wt{\I}_m$
\begin{align*}
\cdots <\ov{\frac{3}{2}}
<\ov{1}<\ov{\hf}<\underbrace{\ov{-1}<\ov{-2}
<\cdots<\ov{-m}}_m<\ov{0}<\underbrace{-m<\cdots<-1}_m
<\hf<1<\frac{3}{2}<\cdots
\end{align*}
For $m\in\Z_+$ define the following subsets of $\wt{\mathbb I}_m$:
\begin{align*}
{{\mathbb I}}_m &:=
\{\underbrace{\ov{-1},\ldots,\ov{-m}}_m,\ov{0},
\underbrace{-m,\ldots,-1}_m\}\cup \{\ov{1},\ov{2},\ov{3},\ldots\}
\cup\{1,2,3,\ldots\},\\
\ov{{\mathbb I}}_m &:=
\{\underbrace{\ov{-1},\ldots,\ov{-m}}_m,\ov{0},
\underbrace{-m,\ldots,-1}_m\}\cup
\{\ov{\frac{1}{2}},\ov{\frac{3}{2}},\ov{\frac{5}{2}},\ldots\}
\cup\{\hf,\frac{3}{2},\frac{5}{2},\ldots\},
 \\
\wt{\I}^+_m &:=\{-m,\ldots,-1,\hf,1,\frac{3}{2},2,\ldots\}.
\end{align*}
For a subset $\X$ of $\wt{\I}_m$, define
$\X^\times:=\X\setminus\{\ov{0}\},\quad \X^+ :=\X\cap\wt{\I}_m^+.
$

\subsubsection{General linear Lie superalgebra}

For a homogeneous element $v$ in a super vector space
$V=V_{\bar{0}}\oplus V_{\bar{1}}$ we denote by $|v|$ its
$\Z_2$-degree.

For $m\in\Z_+$ consider the infinite dimensional super space $\wt{V}_{m}$ over $\C$ with
ordered basis $\{v_i|i\in\wt{\I}_m\}$. We declare $|v_r|=|v_{\ov{r}}|=\bar{0}$, if
$r\in\Z\setminus\{0\}$, and $|v_r|=|v_{\ov{r}}|=\bar{1}$, if $r\in\hf+\Z_+$.  The parity
of the vector $v_{\ov{0}}$ is to be specified. With respect to this basis a linear map on
$\wt{V}_m$ may be identified with a complex matrix $(a_{rs})_{r,s\in\wt{\mathbb{I}}_m}$.
The Lie superalgebra $\gl(\wt{V}_m)$ is the Lie subalgebra of linear transformations on
$\wt{V}_m$ consisting of $(a_{rs})$ with $a_{rs}=0$ for all but finitely many $a_{rs}$'s.
Denote by $E_{rs}\in\gl(\wt{V}_m)$ the elementary matrix with $1$ at the $r$th row and
$s$th column and zero elsewhere.

The vector spaces $V_m$ and $\ov{V}_m$ are defined to be subspaces
of $\wt{V}_m$ with ordered basis $\{v_i\}$ indexed by $\I_m$ and
$\ov{\I}_m$, respectively. The corresponding subspaces of $V_m$,
$\ov{V}_m$ and $\wt{V}_m$ with basis vectors $v_i$, with $i$
indexed by $\I^\times_m$, $\ov{\I}^\times_m$ and
$\wt{\I}^\times_m$, respectively, are denoted by $V^\times_m$,
$\ov{V}^\times_m$ and $\wt{V}^\times_m$, respectively.  This gives
rise to Lie superalgebras $\gl(V_m)$, $\gl(\ov{V}_m)$,
$\gl(V^\times_m)$, $\gl(\ov{V}^\times_m)$ and
$\gl(\wt{V}^\times_m)$.

Let $W$ be one of the spaces $\wt{V}_m,\wt{V}^\times_m, V_m, V_m^\times, \ov{V}_m$ or
$\ov{V}_m^\times$. The standard Cartan subalgebra of $\gl(W)$ is spanned by the basis
$\{E_{rr}\}$, with corresponding dual basis $\{\epsilon_{r}\}$, where $r$ runs over the
index sets $\wt{\I}_m$, $\wt{\I}^\times_m$, ${\I}_m$, ${\I}^\times_m$, $\ov{\I}_m$,
$\ov{\I}^\times_m$, respectively.

\subsubsection{Skew-supersymmetric bilinear form on $W$}
\label{sec:skewsym}

In this subsection we set $|v_{\ov{0}}|=\bar{1}$. For $m\in\Z_+$
define a non-degenerate skew-supersymmetric bilinear form
$(\cdot|\cdot)$ on $\wt{V}_m$ by
\begin{align}\label{symp:bilinear:form}
&(v_{r}|v_{{s}}) =(v_{\ov{r}}|v_{\ov{s}})=0,\quad
(v_r|v_{\ov{s}}) =\delta_{rs}=-(-1)^{|v_r|\cdot|v_s|}(v_{\ov{s}}|v_r),
 \quad r,s\in\wt{\I}_m^+,\\
&(v_{\ov{0}}|v_{\ov{0}})=1, \quad
(v_{\ov{0}}|v_{r})=(v_{\ov{0}}|v_{\ov{r}})=0,\quad
r\in\wt{\I}_m^+.\nonumber
\end{align}
Restricting the form to $\wt{V}^\times_m$, $V_m$, $V^\times_m$,
$\ov{V}_m$ and $\ov{V}^\times_m$ gives rise to non-degenerate
skew-supersymmetric bilinear forms that will again be denoted by
$(\cdot|\cdot)$.

Let $W$ be as before. The Lie superalgebra $\spo(W)$ is the
subalgebra of $\gl(W)$ preserving the bilinear form
$(\cdot|\cdot)$.
The standard Cartan subalgebra of $\spo(W)$ is spanned by the basis
$\{E_r:=E_{rr}-E_{\ov{r},\ov{r}}\}$, with corresponding dual basis $\{\epsilon_r\}$.
We have the following realizations of the corresponding Lie
superalgebras for $\xx=\mf{b^\bullet},\cc$ and $m>0$:

\begin{center}
\begin{table}[ht]
\caption{}\label{table1}
\begin{tabular}{|c|c||c|c|}
    \hline
Lie superalgebra &  Dynkin diagram& Lie superalgebra &  Dynkin diagram\\
    \hline \hline
$\spo(\wt{V}_m)$
 &\makebox(23,0){$\oval(20,14)$}\makebox(-20,8){$\DG^{\mf{b^\bullet}}$}
 &$\spo(\wt{V}^\times_m)$
 & \makebox(23,2){$\oval(20,12)$}\makebox(-20,8){$\DG^{\mf{c}}$}\\
$\spo(V_m)$
 & \makebox(23,-2){$\oval(20,14)$}\makebox(-20,8){$\G^{\mf{b^\bullet}}$}
 &$\spo(V^\times_m)$
 & \makebox(23,1){$\oval(20,12)$}\makebox(-20,8){$\G^{\mf{c}}$}\\
$\spo(\ov{V}_m)$
 & \makebox(23,0){$\oval(20,14)$}\makebox(-20,8){$\SG^{\mf{b^\bullet}}$}
 &$\spo(\ov{V}^\times_m)$
 & \makebox(23,0){$\oval(20,12)$}\makebox(-20,8){$\SG^{\mf{c}}$}\\
    \hline
\end{tabular}
\end{table}
\end{center}
The sets $\Pi^{\xx}$, $\ov{\Pi}^\xx$ and $\wt{\Pi}^\xx$
give rise to the following sets of positive roots:
\begin{align*}
&\wt{\Phi}^{\mf b^\bullet}_+
 =\{\pm\epsilon_r-\epsilon_s|r< s\ (r,s\in\wt{\I}^+_m)\}
 \cup\{-2\epsilon_i\ (i\in\I^+_m)\}\cup\{-\epsilon_r\ (r\in\wt{\I}_m^+)\},\\
&\wt{\Phi}^{\mf c}_+
 =\{\pm\epsilon_r-\epsilon_s|r< s\ (r,s\in\wt{\I}^+_m)\}
 \cup\{-2\epsilon_i\ (i\in\I^+_m)\},\\
&{\Phi}^{\mf b^\bullet}_+
 =\{\pm\epsilon_i-\epsilon_j|i< j\ (i,j\in{\I}^+_m)\}
 \cup\{-\epsilon_i,-2\epsilon_i\ (i\in\I^+_m)\},\\
&{\Phi}^{\mf c}_+
 =\{\pm\epsilon_i-\epsilon_j|i< j\ (i,j\in{\I}^+_m)\}
 \cup\{-2\epsilon_i\ (i\in\I^+_m)\},\\
&\ov{\Phi}^{\mf b^\bullet}_+
 =\{\pm\epsilon_r-\epsilon_s|r< s\ (r,s\in\ov{\I}^+_m)\}
 \cup\{-2\epsilon_i\ (-m\le i\le -1)\}\cup\{-\epsilon_r\ (r\in\ov{\I}_m^+)\},\\
&\ov{\Phi}^{\mf c}_+
 =\{\pm\epsilon_r-\epsilon_s|r< s\ (r,s\in\ov{\I}^+_m)\}
 \cup\{-2\epsilon_i\ (-m\le i\le -1)\}.
\end{align*}
The corresponding subsets of simple roots can be read off from the
corresponding diagrams in \eqref{Dynkin:combined} (here we recall
the notation of roots $\alpha$'s and $\beta$'s from
\eqref{alpha:beta}).

\subsubsection{Supersymmetric bilinear form on $W$}\label{sec:symm}

Let $W$ be as before.  In this subsection we set
$|v_{\ov{0}}|=\bar{0}$. Define a supersymmetric bilinear form
$(\cdot|\cdot)$ on $\wt{V}_m$ by
\begin{align}\label{sym:bilinear:form}
&(v_{r}|v_{{s}})=(v_{\ov{r}}|v_{\ov{s}})=0,\quad
(v_r|v_{\ov{s}})=\delta_{rs}=(-1)^{|v_r|\cdot|v_s|}(v_{\ov{s}}|v_r),
 \quad r,s\in\wt{\I}_m^+,\\
&(v_{\ov{0}}|v_{\ov{0}})=1,\quad (v_{\ov{0}}|v_{r})
=(v_{\ov{0}}|v_{\ov{r}})=0,\quad r\in\wt{\I}_m^+.\nonumber
\end{align}
Restricting the form to $\wt{V}^\times_m$, $V_m$, $V^\times_m$, $\ov{V}_m$ and
$\ov{V}^\times_m$ gives respective non-degenerate supersymmetric bilinear forms that will
also be denoted by $(\cdot|\cdot)$. The Lie superalgebra $\osp(W)$ is the subalgebra of
$\gl(W)$ preserving the respective bilinear form determined by
\eqnref{sym:bilinear:form}. The standard Cartan subalgebra of $\osp(W)$ is also spanned
by the basis $\{E_r:=E_{rr}-E_{\ov{r},\ov{r}}\}$, with corresponding dual basis
$\{\epsilon_r\}$. We have the following realizations of the corresponding Lie
superalgebras for $\xx=\mf{b}, \dd$ and $m>0$:
\begin{center}
\begin{table}[ht]
\caption{}\label{table2}
\begin{tabular}{|c|c||c|c|}
    \hline
Lie superalgebra &  Dynkin diagram& Lie superalgebra &  Dynkin diagram\\
    \hline \hline
$\osp(\wt{V}_m)$
 & \makebox(23,0){$\oval(20,14)$}\makebox(-20,8){$\DG^{\mf{b}}$}
 &$\osp(\wt{V}^\times_m)$
 & \makebox(23,1){$\oval(20,14)$}\makebox(-20,8){$\DG^{\mf{d}}$}\\
$\osp(V_m)$
 & \makebox(23,-2){$\oval(20,13)$}\makebox(-20,8){$\G^{\mf{b}}$}&$\osp(V^\times_m)$
 & \makebox(23,0){$\oval(20,13)$}\makebox(-20,8){$\G^{\mf{d}}$}\\
$\osp(\ov{V}_m)$
 & \makebox(23,0){$\oval(20,14)$}\makebox(-20,8){$\SG^{\mf{b}}$}&$\osp(\ov{V}^\times_m)$
 & \makebox(23,0){$\oval(20,14)$}\makebox(-20,8){$\SG^{\mf{d}}$}\\
    \hline
\end{tabular}
\end{table}
\end{center}
The sets $\Pi^{\xx}$, $\ov{\Pi}^\xx$ and $\wt{\Pi}^\xx$ give rise
to the following sets of positive roots:
\begin{align*}
&\wt{\Phi}^{\mf b}_+
 =\{\pm\epsilon_r-\epsilon_s|r< s\ (r,s\in\wt{\I}^+_m)\}
 \cup\{-2\epsilon_s\ (s\in\ov{\I}^+_0)\}\cup\{-\epsilon_r\ (r\in\wt{\I}_m^+)\},\\
&\wt{\Phi}^{\mf d}_+
 =\{\pm\epsilon_r-\epsilon_s|r< s\ (r,s\in\wt{\I}^+_m)\}
 \cup\{-2\epsilon_s\ (s\in\ov{\I}^+_0)\},\\
&{\Phi}^{\mf b}_+
 =\{\pm\epsilon_i-\epsilon_j|i< j\ (i,j\in{\I}^+_m)\}
 \cup\{-\epsilon_i\ (i\in\I^+_m)\},\\
&{\Phi}^{\mf d}_+
 =\{\pm\epsilon_i-\epsilon_j|i< j\ (i,j\in{\I}^+_m)\},\\
&\ov{\Phi}^{\mf b}_+
 =\{\pm\epsilon_r-\epsilon_s|r< s\ (r,s\in\ov{\I}^+_m)\}
 \cup\{-2\epsilon_s\ (s\in\ov{\I}_0^+)\}\cup\{-\epsilon_r\ (r\in\ov{\I}_m^+)\},\\
&\ov{\Phi}^{\mf d}_+
 =\{\pm\epsilon_r-\epsilon_s|r< s\ (r,s\in\ov{\I}^+_m)\}
 \cup\{-2\epsilon_s\ (s\in\ov{\I}_0^+)\}.
\end{align*}
Again, the subsets of simple roots can be read off from the
corresponding diagrams in \eqref{Dynkin:combined}.

\subsection{The case $m=0$}\label{sec:m=0}

The Dynkin diagrams of $\spo(W)$ with a distinguished set of
simple roots, for $W=\wt{V}_0,\wt{V}_0^\times,V_0, V_0^\times,
\ov{V}_0, \ov{V}_0^\times$ are listed in order as follows (see
also \remref{rem:000} below):
\begin{center}
\hskip -3cm \setlength{\unitlength}{0.16in}
\begin{picture}(24,3)
\put(8,2){\makebox(0,0)[c]{$\bigotimes$}}
\put(10.4,2){\makebox(0,0)[c]{$\bigotimes$}}
\put(14.85,2){\makebox(0,0)[c]{$\bigotimes$}}
\put(17.25,2){\makebox(0,0)[c]{$\bigotimes$}}
\put(19.4,2){\makebox(0,0)[c]{$\bigotimes$}}
\put(5.6,2){\makebox(0,0)[c]{$\bigcirc$}}
\put(8.4,2){\line(1,0){1.55}} \put(10.82,2){\line(1,0){0.8}}
\put(13.2,2){\line(1,0){1.2}} \put(15.28,2){\line(1,0){1.45}}
\put(17.7,2){\line(1,0){1.25}} \put(19.8,2){\line(1,0){1.5}}
\put(6,1.8){$\Longleftarrow$}
\put(12.5,1.95){\makebox(0,0)[c]{$\cdots$}}
\put(22.6,1.95){\makebox(0,0)[c]{$\cdots$}}
\put(8,1){\makebox(0,0)[c]{\tiny $\alpha_{1/2}$}}
\put(10.5,1){\makebox(0,0)[c]{\tiny $\alpha_{1}$}}
\put(14.8,1){\makebox(0,0)[c]{\tiny $\alpha_{r-1/2}$}}
\put(17.15,1){\makebox(0,0)[c]{\tiny $\alpha_{r}$}}
\put(19.8,1){\makebox(0,0)[c]{\tiny $\alpha_{r+1/2}$}}
\put(5.4,1){\makebox(0,0)[c]{\tiny $-\epsilon_{1/2}$}}
\put(0,1.2){{\ovalBox(1.6,1.2){$\DG^{\mf b^\bullet}$}}}
\end{picture}
\end{center}

\begin{center}
\hskip -3cm \setlength{\unitlength}{0.16in}
\begin{picture}(24,4)
\put(8,2){\makebox(0,0)[c]{$\bigotimes$}}
\put(10.4,2){\makebox(0,0)[c]{$\bigotimes$}}
\put(14.85,2){\makebox(0,0)[c]{$\bigotimes$}}
\put(17.25,2){\makebox(0,0)[c]{$\bigotimes$}}
\put(19.4,2){\makebox(0,0)[c]{$\bigotimes$}}
\put(6,3.8){\makebox(0,0)[c]{$\bigotimes$}}
\put(6,.3){\makebox(0,0)[c]{$\bigotimes$}}
\put(8.4,2){\line(1,0){1.55}} \put(10.82,2){\line(1,0){0.8}}
\put(13.2,2){\line(1,0){1.2}} \put(15.28,2){\line(1,0){1.45}}
\put(17.7,2){\line(1,0){1.25}} \put(19.8,2){\line(1,0){1.5}}
\put(7.5,2.2){\line(-1,1){1.3}}
\put(7.6,1.8){\line(-1,-1){1.25}}
\put(6,.8){\line(0,1){2.6}}
\put(12.5,1.95){\makebox(0,0)[c]{$\cdots$}}
\put(22.6,1.95){\makebox(0,0)[c]{$\cdots$}}
\put(8.1,1){\makebox(0,0)[c]{\tiny $\alpha_{1}$}}
\put(10.9,1){\makebox(0,0)[c]{\tiny $\alpha_{3/2}$}}
\put(14.8,1){\makebox(0,0)[c]{\tiny $\alpha_{r-1/2}$}}
\put(17.15,1){\makebox(0,0)[c]{\tiny $\alpha_{r}$}}
\put(19.8,1){\makebox(0,0)[c]{\tiny $\alpha_{r+1/2}$}}
\put(4.5,3.8){\makebox(0,0)[c]{\tiny $\alpha_{1/2}$}}
\put(3.5,0.3){\makebox(0,0)[c]{\tiny $-\epsilon_{1/2}-\epsilon_{1}$}}
\put(0,1.2){{\ovalBox(1.6,1.2){$\DG^{\mf c}$}}}
\end{picture}
\end{center}

\begin{center}
\hskip -3cm \setlength{\unitlength}{0.16in}
\begin{picture}(24,3)
\put(8,2){\makebox(0,0)[c]{$\bigcirc$}}
\put(10.4,2){\makebox(0,0)[c]{$\bigcirc$}}
\put(14.85,2){\makebox(0,0)[c]{$\bigcirc$}}
\put(17.25,2){\makebox(0,0)[c]{$\bigcirc$}}
\put(19.4,2){\makebox(0,0)[c]{$\bigcirc$}}
\put(5.6,2){\circle*{0.9}}
\put(8.4,2){\line(1,0){1.55}} \put(10.82,2){\line(1,0){0.8}}
\put(13.2,2){\line(1,0){1.2}} \put(15.28,2){\line(1,0){1.45}}
\put(17.7,2){\line(1,0){1.25}} \put(19.8,2){\line(1,0){1.5}}
\put(6,1.8){$\Longleftarrow$}
\put(12.5,1.95){\makebox(0,0)[c]{$\cdots$}}
\put(22.6,1.95){\makebox(0,0)[c]{$\cdots$}}
\put(8,1){\makebox(0,0)[c]{\tiny $\beta_1$}}
\put(10.5,1){\makebox(0,0)[c]{\tiny $\beta_{2}$}}
\put(14.8,1){\makebox(0,0)[c]{\tiny $\beta_{n-1}$}}
\put(17.15,1){\makebox(0,0)[c]{\tiny $\beta_{n}$}}
\put(19.8,1){\makebox(0,0)[c]{\tiny $\beta_{n+1}$}}
\put(5.4,1){\makebox(0,0)[c]{\tiny $-\epsilon_{1}$}}
\put(-1,1.2){{\ovalBox(1.6,1.2){$\G^{\mf b^\bullet}$}}}
\put(1,1.5){=}
\put(2,1.2){{\ovalBox(1.6,1.2){$\SG^{\mf b}$}}}
\end{picture}
\end{center}

\begin{center}
\hskip -3cm \setlength{\unitlength}{0.16in}
\begin{picture}(24,3)
\put(8,2){\makebox(0,0)[c]{$\bigcirc$}}
\put(10.4,2){\makebox(0,0)[c]{$\bigcirc$}}
\put(14.85,2){\makebox(0,0)[c]{$\bigcirc$}}
\put(17.25,2){\makebox(0,0)[c]{$\bigcirc$}}
\put(19.4,2){\makebox(0,0)[c]{$\bigcirc$}}
\put(5.6,2){\makebox(0,0)[c]{$\bigcirc$}}
\put(8.4,2){\line(1,0){1.55}} \put(10.82,2){\line(1,0){0.8}}
\put(13.2,2){\line(1,0){1.2}} \put(15.28,2){\line(1,0){1.45}}
\put(17.7,2){\line(1,0){1.25}} \put(19.8,2){\line(1,0){1.5}}
\put(6,1.8){$\Longrightarrow$}
\put(12.5,1.95){\makebox(0,0)[c]{$\cdots$}}
\put(22.6,1.95){\makebox(0,0)[c]{$\cdots$}}
\put(8,1){\makebox(0,0)[c]{\tiny $\beta_1$}}
\put(10.5,1){\makebox(0,0)[c]{\tiny $\beta_{2}$}}
\put(14.8,1){\makebox(0,0)[c]{\tiny $\beta_{n-1}$}}
\put(17.15,1){\makebox(0,0)[c]{\tiny $\beta_{n}$}}
\put(19.8,1){\makebox(0,0)[c]{\tiny $\beta_{n+1}$}}
\put(5.4,1){\makebox(0,0)[c]{\tiny $-2\epsilon_{1}$}}
\put(-1,1.2){{\ovalBox(1.6,1.2){$\G^{\mf c}$}}}
\put(1,1.5){=}
\put(2,1.2){{\ovalBox(1.6,1.2){$\SG^{\mf d}$}}}
\end{picture}
\end{center}

\begin{center}
\hskip -3cm \setlength{\unitlength}{0.16in}
\begin{picture}(24,3)
\put(8,2){\makebox(0,0)[c]{$\bigcirc$}}
\put(10.4,2){\makebox(0,0)[c]{$\bigcirc$}}
\put(14.85,2){\makebox(0,0)[c]{$\bigcirc$}}
\put(17.25,2){\makebox(0,0)[c]{$\bigcirc$}}
\put(19.4,2){\makebox(0,0)[c]{$\bigcirc$}}
\put(5.6,2){\makebox(0,0)[c]{$\bigcirc$}}
\put(8.4,2){\line(1,0){1.55}} \put(10.82,2){\line(1,0){0.8}}
\put(13.2,2){\line(1,0){1.2}} \put(15.28,2){\line(1,0){1.45}}
\put(17.7,2){\line(1,0){1.25}} \put(19.8,2){\line(1,0){1.5}}
\put(6,1.8){$\Longleftarrow$}
\put(12.5,1.95){\makebox(0,0)[c]{$\cdots$}}
\put(22.6,1.95){\makebox(0,0)[c]{$\cdots$}}
\put(8,1){\makebox(0,0)[c]{\tiny $\beta_{1/2}$}}
\put(10.5,1){\makebox(0,0)[c]{\tiny $\beta_{3/2}$}}
\put(14.8,1){\makebox(0,0)[c]{\tiny $\beta_{r-1}$}}
\put(17.15,1){\makebox(0,0)[c]{\tiny $\beta_{r}$}}
\put(19.8,1){\makebox(0,0)[c]{\tiny $\beta_{r+1}$}}
\put(5.4,1){\makebox(0,0)[c]{\tiny $-\epsilon_{1/2}$}}
\put(-1,1.2){{\ovalBox(1.6,1.2){$\SG^{\mf b^\bullet}$}}}
\put(1,1.5){=}
\put(2,1.2){{\ovalBox(1.6,1.2){$\G^{\mf b}$}}}
\end{picture}
\end{center}

\begin{center}
\hskip -3cm \setlength{\unitlength}{0.16in}
\begin{picture}(24,4)
\put(8,2){\makebox(0,0)[c]{$\bigcirc$}}
\put(10.4,2){\makebox(0,0)[c]{$\bigcirc$}}
\put(14.85,2){\makebox(0,0)[c]{$\bigcirc$}}
\put(17.25,2){\makebox(0,0)[c]{$\bigcirc$}}
\put(19.4,2){\makebox(0,0)[c]{$\bigcirc$}}
\put(6,3.8){\makebox(0,0)[c]{$\bigcirc$}}
\put(6,.3){\makebox(0,0)[c]{$\bigcirc$}}
\put(8.4,2){\line(1,0){1.55}} \put(10.82,2){\line(1,0){0.8}}
\put(13.2,2){\line(1,0){1.2}} \put(15.28,2){\line(1,0){1.45}}
\put(17.7,2){\line(1,0){1.25}} \put(19.8,2){\line(1,0){1.5}}
\put(7.6,2.2){\line(-1,1){1.3}}
\put(7.6,1.8){\line(-1,-1){1.3}}
\put(12.5,1.95){\makebox(0,0)[c]{$\cdots$}}
\put(22.6,1.95){\makebox(0,0)[c]{$\cdots$}}
\put(8.1,1){\makebox(0,0)[c]{\tiny $\beta_{3/2}$}}
\put(10.7,1){\makebox(0,0)[c]{\tiny $\beta_{5/2}$}}
\put(14.8,1){\makebox(0,0)[c]{\tiny $\beta_{r-1}$}}
\put(17.15,1){\makebox(0,0)[c]{\tiny $\beta_{r}$}}
\put(19.8,1){\makebox(0,0)[c]{\tiny $\beta_{r+1}$}}
\put(4.5,3.8){\makebox(0,0)[c]{\tiny $\beta_{1/2}$}}
\put(3.5,0.3){\makebox(0,0)[c]{\tiny $-\epsilon_{1/2}-\epsilon_{3/2}$}}
\put(-1,1.2){{\ovalBox(1.6,1.2){$\SG^{\mf c}$}}}
\put(1,1.5){=}
\put(2,1.2){{\ovalBox(1.6,1.2){$\G^{\mf d}$}}}
\end{picture}
\end{center}
For the sake of completeness we also list the corresponding sets
of positive roots.
\begin{align*}
&\wt{\Phi}^{\mf b^\bullet}_+
 =\{\pm\epsilon_r-\epsilon_s|r< s\ (r,s\in\wt{\I}^+_0)\}
 \cup\{-2\epsilon_i\ (i\in\I^+_0)\}\cup\{-\epsilon_r\ (r\in\wt{\I}_0^+)\},\\
&\wt{\Phi}^{\mf c}_+
  =\{\pm\epsilon_r-\epsilon_s|r< s\ (r,s\in\wt{\I}^+_0)\}
  \cup\{-2\epsilon_i\ (i\in\I^+_0)\},\\
&{\Phi}^{\mf b^\bullet}_+
 =\{\pm\epsilon_i-\epsilon_j|i< j\ (i,j\in{\I}^+_0)\}
 \cup\{-\epsilon_i,-2\epsilon_i\ (i\in\I^+_0)\},\\
&{\Phi}^{\mf c}_+
 =\{\pm\epsilon_i-\epsilon_j|i< j\ (i,j\in{\I}^+_0)\}
 \cup\{-2\epsilon_i\ (i\in\I^+_0)\},\\
&\ov{\Phi}^{\mf b^\bullet}_+
 =\{\pm\epsilon_r-\epsilon_s|r< s\ (r,s\in\ov{\I}^+_0)\}
 \cup\{-\epsilon_r\ (r\in\ov{\I}_0^+)\},\\
&\ov{\Phi}^{\mf c}_+
 =\{\pm\epsilon_r-\epsilon_s|r< s\ (r,s\in\ov{\I}^+_0)\}.
\end{align*}

\begin{rem}\label{rem:000}
It is easy to see that we have the following isomorphisms of Lie
superalgebras with identical Dynkin diagrams:
$\osp(V_0)\cong\spo(\ov{V}_0)$, $\osp(\ov{V}_0) \cong\spo(V_0)$,
$\osp(V^\times_0)\cong\spo(\ov{V}^\times_0)$,
$\osp(\ov{V}^\times_0)\cong\spo(V^\times_0)$.
\end{rem}

\subsection{Central extensions}  \label{sec:ext}
We will replace the above matrix realization of the Lie
superalgebras with Dynkin diagrams
\makebox(23,0){$\oval(20,14)$}\makebox(-20,8){$\G^\xx$},
\makebox(23,0){$\oval(20,14)$}\makebox(-20,8){$\SG^\xx$} and
\makebox(23,0){$\oval(20,14)$}\makebox(-20,8){$\DG^\xx$} by their
central extensions, for $\xx=\mf{b, b^\bullet, c, d}$. These
central extensions will be convenient and conceptual for later
formulation of truncation functors and super duality.

Let $m\in\Z_+$. Consider the central extension
$\widehat{\gl}(\wt{V}_m)$ of $\gl(\wt{V}_m)$
by the one-dimensional center $\C K$ determined by the $2$-cocycle
\begin{align*}
\tau(A,B):=\text{Str}([\mathfrak{J},A]B),\quad A,B\in\gl(\wt{V}_m),
\end{align*}
where $\mathfrak{J}=E_{\ov{0}\ov{0}}+\sum_{r\le \ov{\hf}}E_{rr}$
and $\text{Str}$ denotes the supertrace. Observe that the cocycle
$\tau$ is a coboundary.
Indeed, as a vector space, $\widehat{\gl}(\wt{V}_m) =\gl(\wt{V}_m) \oplus \C K$,
and let us denote by $\widehat{X}$ for $X\in\gl(\wt{V}_m)$ to indicate
that it is in $\widehat{\gl}(\wt{V}_m)$. Then the map from $\widehat{\gl}(\wt{V}_m)$
to the direct sum of Lie superalgebras $\gl(\wt{V}_m) \oplus \C K$, which sends
$\widehat{X}$ to ${X}':= X - \text{Str}(\mf{J}X) K,$
 is an algebra isomorphism, i.e., $[{X}', {Y}']= {[X,Y]}' +\tau(X,Y) K$.

For $W=\wt{V}^\times_m, V_m, V_m^\times, \ov{V}_m,
\ov{V}_m^\times$ the restrictions of $\tau$ to the subalgebras
$\gl(W)$ give rise to respective central extensions, which in turn
induce central extensions on $\osp(W)$ and $\spo(W)$. We denote
such a central extension of $\spo(W)$ or $\osp(W)$ by $\G^\xx$
(respectively, $\SG^\xx$ and $\DG^\xx$) when it corresponds to the
Dynkin diagram
\makebox(23,0){$\oval(20,13)$}\makebox(-20,8){$\G^\xx$} in Tables
\ref{table1} and \ref{table2} (respectively,
\makebox(23,0){$\oval(20,13)$}\makebox(-20,8){$\SG^\xx$} and
\makebox(23,0){$\oval(20,14)$}\makebox(-20,8){$\DG^\xx$}).

We make a trivial yet crucial observation that $\G^\xx$ and $\SG^\xx$ are naturally
subalgebras of $\DG^\xx$. The standard Cartan subalgebras of $\G^\xx$, $\SG^\xx$ and
$\DG^\xx$ will be denoted by $\h^\xx$, $\ov{\h}^\xx$ and $\wt{\h}^\xx$, respectively.
$\h^\xx$, $\ov{\h}^\xx$ or $\wt{\h}^\xx$ has a basis $\{K,\widehat{E}_{r} \}$  with dual
basis $\{\Lambda_0,\epsilon_r\}$ in the restricted dual $(\h^\xx)^*$, $(\ov{\h}^\xx)^*$
or $(\wt{\h}^\xx)^*$,  where $r$ runs over the index sets $\I^+_m$, $\ov{\I}^+_m$ or
$\wt{\I}^+_m$, respectively. Here $\La_0$ is defined by letting
\begin{align*}
\La_0(K)=1,\quad\La_0(\widehat{E}_{r})=0,
\end{align*}
for all relevant $r$ in each case.

{\bf In the remainder of the paper we shall drop the superscript
$\xx$.} For example, we write $\G$, $\SG$ and $\wt{\G}$ for
$\G^\xx$, $\SG^\xx$ and $\wt{\G}^\xx$, with associated Dynkin
diagrams \makebox(23,0){$\oval(20,14)$}\makebox(-20,8){$\G$},
\makebox(23,0){$\oval(20,14)$}\makebox(-20,8){$\SG$} and
\makebox(23,0){$\oval(20,14)$}\makebox(-20,8){$\DG$},
respectively, where $\xx$ denotes a fixed type among
$\mf{b,b^\bullet,c,d}$.

\section{Categories $\mathcal{O}$, $\ov{\mathcal{O}}$
and $\wt{\mathcal{O}}$}\label{sec:O}

In this section, we first introduce the categories $\mathcal{O}$, $\ov{\mathcal{O}}$ and
$\wt{\mathcal{O}}$ of $\G$-modules, $\SG$-modules and $\wt{\G}$-modules, respectively.
Then we study the truncation functors which relate $\SG$ to finite dimensional Lie
superalgebras of $\osp$ type.

Let $m\in\Z_+$ 
be fixed.

\subsection{The weights}
\label{sec:OBCD}

We fix an
arbitrary subset $Y_0$ of $\Pi({\mf k})$.

Let $Y$, $\ov{Y}$ and $\wt{Y}$ be the union of $Y_0$ and the subset of simple roots of
\makebox(23,0){$\oval(20,12)$}\makebox(-20,8){$\mf T$},
\makebox(23,0){$\oval(20,12)$}\makebox(-20,8){$\ov{\mf T}$}  and
\makebox(23,0){$\oval(20,13)$}\makebox(-20,8){$\wt{\mf T}$}, respectively, with the
leftmost one removed. We have $Y_0=\emptyset$ for $m=0$. As $Y$, $\ov{Y}$ and $\wt{Y}$
are fixed, we will make the convention of suppressing them from notations below. Set
$\mf{l}$, $\ov{\mf{l}}$ and $\wt{\mf l}$ to be the standard Levi subalgebras of $\G$,
$\SG$ and $\DG$ corresponding to the subsets $Y$, $\ov{Y}$ and $\wt{Y}$, respectively.
The Borel subalgebras of $\G$, $\SG$ and $\DG$, spanned by the central element $K$ and
upper triangular matrices, are denoted by $\mf{b}$, $\ov{\mf{b}}$ and $\wt{\mf{b}}$,
respectively. Let $\mf{p} =\mf{l} +\mf{b}$, $\ov{\mf p} =\ov{\mf{l}} + \ov{\mf{b}}$ and
$\wt{\mf p} =\wt{\mf{l}} +\wt{\mf{b}}$ be the corresponding parabolic subalgebras with
nilradicals $\mf{u}$, $\ov{\mf u}$ and $\wt{\mf u}$ and opposite nilradicals $\mf{u}_-$,
$\ov{\mf u}_-$ and $\wt{\mf u}_-$, respectively.

Given a partition $\mu=(\mu_1,\mu_2,\ldots)$, we denote by
$\ell(\mu)$ the length of $\mu$ and by $\mu'$ its conjugate
partition. We also denote by $\theta(\mu)$ the modified Frobenius
coordinates of $\mu$:
\begin{equation*}
\theta(\mu)
:=(\theta(\mu)_{1/2},\theta(\mu)_1,\theta(\mu)_{3/2},\theta(\mu)_2,\ldots),
\end{equation*}
where
$$\theta(\mu)_{i-1/2}:=\max\{\mu'_i-i+1,0\}, \quad
\theta(\mu)_i:=\max\{\mu_i-i, 0\}, \quad i\in\N.
$$

Let $\la_{-m},\ldots,\la_{-1}\in\C$ and $\la^+$ be a partition.
The tuple $(\la_{-m},\ldots,\la_{-1};\la^+)$ is said to satisfy a
dominant condition if $\langle\sum_{i=-m}^{-1}\la_i\epsilon_i,
h_\alpha \rangle\in\Z_+$ for all $\alpha\in Y_0$, where $h_\alpha$
denotes the coroot of $\alpha$. Associated to such a dominant
tuple and each $d \in \C$, we define the weights (which will be
called {\em dominant})
\begin{align}
{\la} &:=\sum_{i=-m}^{-1}\la_{i}\epsilon_{i}
 + \sum_{j\in\N}\la^+_{j}\epsilon_{j}
 + d\La_0\in {\h}^{*},\label{weight:Im}\\
\la^\natural &:=\sum_{i=-m}^{-1}\la_{i}\epsilon_{i}
 + \sum_{s\in\hf+\Z_+}(\la^+)'_{s+\hf}\epsilon_s
 + d\La_0\in \ov{\h}^{*}, \label{weight:ovIm}\\
\la^\theta &:=\sum_{i=-m}^{-1}\la_{i}\epsilon_{i}
 + \sum_{r\in\hf\N}\theta(\la^+)_r\epsilon_r
 + d\La_0\in \wt{\h}^{*}.\label{weight:wtIm}
\end{align}
We denote by $P^+\subset{\h}^*$, $\bar{P}^+\subset\ov{\h}^*$ and
$\wt{P}^+\subset\wt{\h}^*$ the sets of all dominant weights of the form
\eqnref{weight:Im}, \eqnref{weight:ovIm} and \eqnref{weight:wtIm} for all $d\in \C$,
respectively. By definition we have bijective maps
\begin{eqnarray*}
\natural:P^+\longrightarrow\bar{P}^+, && \quad \la\mapsto
\la^\natural, \\
\theta:P^+\longrightarrow\wt{P}^+, && \quad \la\mapsto \la^\theta.
\end{eqnarray*}

For $\mu\in P^+$, let $L(\mf{l},\mu)$ denote the highest weight
irreducible $\mf{l}$-module of highest weight $\mu$. We extend
$L(\mf{l},\mu)$ to a $\mf{p}$-module by letting $\mf{u}$ act
trivially.  Define as usual the parabolic Verma module
$\Delta(\mu)$ and its irreducible quotient $L(\mu)$ over $\G$:
\begin{align*}
\Delta(\mu):=\text{Ind}_{\mf{p}}^{\G}L(\mf{l},\mu), \qquad
\Delta(\mu) \twoheadrightarrow L(\mu).
\end{align*}
Similarly, for $\mu\in{P^+}$, we define the irreducible $\ov{\mf{l}}$-module
$L(\ov{\mf{l}},\mu^\natural)$, the parabolic Verma $\SG$-module
$\ov{\Delta}(\mu^\natural)$ and its irreducible $\SG$-quotient $\ov{L}(\mu^\natural)$, as
well as the irreducible $\wt{\mf{l}}$-module $L(\wt{\mf{l}},\mu^\theta)$, the parabolic
Verma $\DG$-module $\wt{\Delta}(\mu^\theta)$ and its irreducible $\DG$-quotient
$\wt{L}(\mu^\theta)$.

\subsection{The categories $\mc{O}, \ov{\mc O}$ and $\wt{\mc O}$}

\begin{lem}\label{lem:paraverma}
Let $\mu\in P^+$.
\begin{itemize}
\item[(i)] The restrictions to $\mf{l}$ of the $\G$-modules
$\Delta(\mu)$ and $L(\mu)$  decompose  into direct sums of
$L(\mf{l},\nu)$ for $\nu\in P^+$.

\item[(ii)] The restrictions to $\ov{\mf{l}}$ of the $\SG$-modules
$\ov{\Delta}(\mu^\natural)$ and $\ov{L}(\mu^\natural)$ decompose
into direct sums of $L(\ov{\mf{l}},\nu^\natural)$ for $\nu\in
P^+$.

\item[(iii)] The restrictions to $\wt{\mf{l}}$ of the
$\DG$-modules $\wt{\Delta}(\mu^\theta)$ and $\wt{L}(\mu^\theta)$
decompose into direct sums of $L(\wt{\mf{l}},\nu^\theta)$ for
$\nu\in P^+$.
\end{itemize}
\end{lem}

\begin{proof}
Part (i) is clear.

The proofs of (ii) and (iii) are analogous, and so we shall only
give the proof for (ii). The $\ov{\mf{l}}$-module $\ov{\mf{u}}_-$
is a direct sum of irreducible modules of the form
$L(\ov{\mf{l}},\nu^\natural)$. Now the category of
$\ov{\mf{l}}$-modules that have an increasing composition series
with composition factors isomorphic to
$L(\ov{\mf{l}},\nu^\natural)$, with $\nu\in P^+$, is a semi-simple
tensor category \cite[Section 3.2]{CK}.  Thus
$\ov{\Delta}(\mu^\natural)\cong U{(}\ov{\mf{u}}_-{)}\otimes
L(\ov{\mf{l}},\mu^\natural)$ also decomposes into a direct sum of
$L(\ov{\mf{l}},\nu^\natural)$ with $\nu\in P^+$, and so does its
irreducible quotient $\ov{L}(\mu^\natural)$.
\end{proof}

Let $\mc{O}$ be the category of $\G$-modules ${M}$ such that ${M}$ is a semisimple
${\h}$-module with finite dimensional weight subspaces $M_\gamma$, $\gamma\in {\h}^*$,
satisfying
\begin{itemize}
\item[(i)] ${M}$ decomposes over ${\mf{l}}$ into a direct sum of
$L({\mf{l}},\mu)$ for $\mu\in {P^+}$.

\item[(ii)] There exist finitely many weights
$\la_1,\la_2,\ldots,\la_k\in{P^+}$ (depending on ${M}$) such that
if $\gamma$ is a weight in ${M}$, then
$\gamma\in\la_i-\sum_{\alpha\in{\Pi}}\Z_+\alpha$, for some $i$.
\end{itemize}
The parabolic Verma modules $\Delta(\mu)$ and irreducible modules
$L(\mu)$ for $\mu\in P^+$ lie in $\mc{O}$, by
\lemref{lem:paraverma}. Analogously we define the categories
$\ov{\mc{O}}$
and $\wt{\mc{O}}$
of $\SG$- and $\DG$-modules, respectively. They also contain suitable parabolic Verma and
irreducible modules. The morphisms in $\mc{O}$, $\ov{\mc{O}}$ and $\wt{\mc{O}}$ are all
(not necessarily even) $\G$-, ${\SG}$- and ${\DG}$-homomorphisms, respectively.

\subsection{The Lie superalgebras $\G_n$, $\SG_n$ and $\DG_n$ of finite rank}
\label{sec:finiterank}

For $n\in\N$, recall the sets $\Pi_n, \ov{\Pi}_n, \wt{\Pi}_n$ of
simple roots for the Dynkin diagrams \eqref{Dynkin:combined}.
The associated Lie superalgebras $\G_n$, $\SG_n$ and $\DG_n$ can be identified naturally
with the subalgebras of $\G$, $\SG$ and $\DG$ generated by $K$ and the root vectors of
the corresponding Dynkin diagrams in (\ref{Dynkin:combined}), and moreover, $\G_n \subset
\G_{n+1}, \SG_n \subset \SG_{n+1}$ for all $n$. Observe that the $\SG_n$'s (modulo the
trivial central extensions) are exactly all the finite dimensional Lie superalgebras of
$\osp$ type. Since $\SG =\cup_n \SG_n$, the standard Cartan subalgebra of $\SG_n$ equals
$\ov{\h}_n=\ov{\h}\cap\SG_n$. Similarly, we use the notation ${\h}_n$ and $\wt{\h}_n$ for
the standard Cartan subalgebras of $\G_n$ and $\DG_n$, respectively.

Recall the notation $\la \in P^+$, $\la^\natural$, and
$\la^\theta$ from \eqnref{weight:Im}, \eqnref{weight:ovIm} and
\eqnref{weight:wtIm}. Given $\la\in P^+$ with $\la^+_j=0$ for
$j>n$, we may regard it as a weight $\la_n \in \h^*_n$ in a
natural way. Similarly, for $\la\in P^+$  with $(\la^+)'_j=0$ for
$j>n$, we regard $\la^\natural$ as a weight $\la^\natural_n \in
\ov{\h}^*_n$. Finally, for $\la\in P^+$ with $\theta(\la^+)_j=0$
for $j>n$, we regard $\la^\theta$ as a weight $\la_n^\theta \in
\wt{\h}^*_n$. The subsets of such weights $\la_n, \la_n^\natural,
\la_n^\theta$ in $\h^*_n$, $\ov{\h}^*_n$ and $\wt{\h}^*_n$ will be
denoted by $P^+_n$, $\bar{P}^+_n$ and $\wt{P}^+_n$, respectively.

The corresponding parabolic Verma and irreducible $\G_n$-modules
are denoted by $\Delta_n (\mu)$ and $L_n(\mu)$, respectively, with
$\mu\in P^+_n$, while the corresponding category of $\G_n$-modules
is denoted by $\mc{O}_n$.
Similarly, we introduce the self-explanatory notations
$\ov{\Delta}_n(\mu^\natural)$, $\ov{L}_n(\mu^\natural)$,
$\ov{\mc{O}}_{n}$, and $\wt{\Delta}_n(\mu^\theta)$,
$\ov{L}_n(\mu^\theta)$, $\wt{\mc{O}}_{n}$ for $\SG_n$- and
$\DG_n$-modules, respectively.

\subsection{The truncation functors}

Let $\infty\ge k>n$. For $M\in \mc{O}_k$, we can write
$M=\bigoplus_{\gamma}M_\gamma$, where $\gamma$ runs over
$\gamma\in\sum_{i=-m}^{-1}\C \epsilon_i+\sum_{0<j\le
k}\C \epsilon_j+ \C\La_0$. The {\em truncation functor}
$$
\mf{tr}^{k}_n:\mc{O}_k \rightarrow\mc{O}_n
$$
is defined by sending $M$ to
$\bigoplus_{\nu}M_\nu$, summed over
$\sum_{i=-m}^{-1}\C \epsilon_i+\sum_{0<j\le n}\C\epsilon_j+ \C\La_0$.
When it is clear from the context we shall also write $\mf{tr}_n$
instead of $\mf{tr}^k_n$. Analogously, truncation functors
$\mf{tr}^{k}_n:\ov{\mc{O}}_k\rightarrow\ov{\mc{O}}_n $ and
$\mf{tr}^{k}_n:\wt{\mc{O}}_k\rightarrow\wt{\mc{O}}_n$ are defined.

\begin{lem}\label{lem:trunc}
Let $\infty\ge k>n$ and $X=L,\Delta$.
\begin{itemize}
\item[(i)] For $\mu\in P^+_k$ we have $\mf{tr}_n\big{(}X_k
(\mu)\big{)} =
\begin{cases}
X_n (\mu),\quad\text{if }
\langle\mu,\widehat{E}_j\rangle=0,\forall j>n,\\
0,\quad\text{otherwise}.
\end{cases}$

\item[(ii)] For $\mu\in\bar{P}^+_k$ we have
$\mf{tr}_n\big{(}\ov{X}_k (\mu)\big{)}=
\begin{cases}
\ov{X}_n(\mu),\quad\text{if }
\langle\mu,\widehat{E}_j\rangle=0,\forall j>n,\\
0,\quad\text{otherwise}.
\end{cases}$

\item[(iii)] For $\mu\in\wt{P}^+_k$ we have
$\mf{tr}_n\big{(}\wt{X}_k(\mu)\big{)}=
\begin{cases}
\wt{X}_n(\mu),\quad\text{if }\langle\mu,\widehat{E}_j\rangle=0,\forall j>n,\\
0,\quad\text{otherwise}.
\end{cases}$
\end{itemize}
\end{lem}

\begin{proof}
We will show (i) only. The proofs of (ii) and (iii) are similar.

Since $\mf{tr}^k_n\circ\mf{tr}^l_k=\mf{tr}^l_n$, it is enough to show (i) for $k=\infty$.
Suppose that $\langle\mu,\widehat{E}_{j}\rangle=0$ for all $j>n$. Let $\mf {l}'$ denote
the standard Levi subalgebra of $\G$ corresponding to the removal of the vertex
$\beta_{n} $ of the Dynkin diagram of $\G$. Then $\mf{l}'\cong \G_n\oplus
\mf{gl}(\infty)$. Now $L(\mu)$ is the unique irreducible quotient of the $\G$-module
obtained via parabolic induction from the $\mf{l}'$-module $L_n(\mu)$ (where the
$\G_n$-module $L_n(\mu)$ is extended to an $\mf{l}'$-module by a trivial action of
$\mf{gl}(\infty)$). Our choice of the Levi subalgebra and of the opposite nilradical
assures that this parabolically induced module truncates to $L_n(\mu)$. Thus its
irreducible quotient $L(\mu)$ also truncates to $L_n(\mu)$. The remaining case in (i) is
clear.
\end{proof}

\begin{rem}
The central extensions introduced in Section \ref{sec:ext} allow us to study, in a
uniform fashion, modules whose weights stabilize at any $d\in\C$ (not just at $d=0$). For
example, if $\mu$ is a weight with $\mu(E_{r})=d\not=0$, for $r\gg 0$, then, without
central extensions, the usual truncation functors would always truncate an irreducible or
parabolic Verma of such a highest weight to zero. A way around central extensions is to
define truncation functors depending on each $d\in\C$. This approach, although
equivalent, looks less elegant.
\end{rem}

\section{The character formulas}\label{sec:character}

In this section we introduce two functors $T: \wt{\mc
O} \rightarrow \mc O$ and $\ov{T}: \wt{\mc O} \rightarrow \ov{\mc
O}$,
and then establish a fundamental property of these functors
(Theorem~\ref{matching:modules}). As a consequence, we obtain an
irreducible $\osp$-character formula in a parabolic category $O$
in terms of the KL polynomials of $BCD$ types
(Theorem~\ref{character} and Remark~\ref{rem:sol}).
The Kazhdan-Lusztig polynomials
in the categories $\mc O$, $\ov{\mc O}$ and $\wt{\mc O}$ in terms
of Kostant $\mf u$-homology groups are shown
to match perfectly with one another (\thmref{matching:KLpol}).

\subsection{Odd reflections}\label{odd}

Let $\mc G$ be a Lie superalgebra with a Borel subalgebra $\mc{B}$
with corresponding sets of simple and positive roots $\Pi(\mc{B})$
and $\Phi_+(\mc{B})$, respectively. As usual, for a positive root
$\beta$, we let $f_\beta$ denote a root vector associated to root
$-\beta$.

Let $\alpha$ be an isotropic odd simple root in  $\Pi(\mc{B})$ and $h_\alpha$ be its
corresponding coroot. The set $\Phi_+(\mc{B}^\alpha)
:=\{-\alpha\}\cup\Phi_+(\mc{B})\setminus\{\alpha\}$ forms a new
set of positive roots whose corresponding set of simple roots is
\begin{align*}
\Pi(\mc{B}^\alpha)
=\{\beta\in\Pi(\mc{B})|\langle\beta,h_\alpha\rangle=0, \beta \neq
\alpha\}\cup
\{\beta+\alpha|\beta\in\Pi(\mc{B}),\langle\beta,h_\alpha
\rangle\not=0\} \cup\{-\alpha\}.
\end{align*}
We shall denote by $\mc B^\alpha$ the corresponding new Borel
subalgebra. The process of such a change of Borel subalgebras is
referred to as {\em odd reflection} with respect to $\alpha$
\cite{LSS}.

The following simple and fundamental lemma for odd reflections has
been used by many authors (cf. e.g. \cite{KW}).

\begin{lem}  \label{hwt odd}
Let $L$ be a simple $\mc G$-module of $\mc B$-highest weight $\la$
and let $v$ be a $\mc B$-highest weight vector. Let $\alpha$ be a simple
isotropic odd root in  $\Pi(\mc{B})$.
\begin{enumerate}
\item If $\langle \la, h_\alpha \rangle = 0$, then $L$ is a $\mc
G$-module of $\mc B^\alpha$-highest weight $\la$  and $v$ is a
$\mc B^\alpha$-highest weight vector.

\item If $\langle \la, h_\alpha \rangle \neq 0$, then $L$ is a
$\mc G$-module of $\mc B^\alpha$-highest weight $\la -\alpha$ and
$f_\alpha v$ is a $\mc B^\alpha$-highest weight vector.
\end{enumerate}
\end{lem}

\subsection{The Borel subalgebras $\wt{\mf{b}}^{c}(n)$ and $\wt{\mf{b}}^{s}(n)$}
\label{oddreflection}


Fix $n\in \N$.
Starting with the third Dynkin diagram in \eqref{Dynkin:combined}
associated to $\DG$, we apply the following sequence of
$\frac{n(n+1)}{2}$ odd reflections. First we apply one odd
reflection corresponding to $\alpha_{1/2}$, then we apply two odd
reflections corresponding to $\alpha_{3/2}$ and
$\alpha_{1/2}+\alpha_1+\alpha_{3/2}$.  After that we apply three
odd reflections corresponding to $\alpha_{5/2}$,
$\alpha_{3/2}+\alpha_{2}+\alpha_{5/2}$, and
$\alpha_{1/2}+\alpha_1+\alpha_{3/2}+\alpha_{2}+\alpha_{5/2}$, et
cetera, until finally we apply $n$ odd reflections corresponding
to $\alpha_{n-1/2},\alpha_{n-3/2}+\alpha_{n-1}
+\alpha_{n-1/2},\ldots,\sum_{i=1}^{2n-1}\alpha_{i/2}$. The
resulting new Borel subalgebra for $\DG$ will be denoted by
$\wt{\mf{b}}^{c}(n)$ and its corresponding simple roots are listed in
the following Dynkin diagram:
%
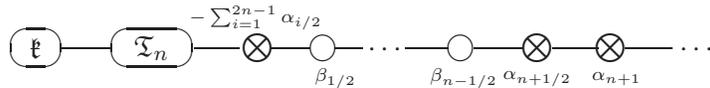
\begin{center}
\hskip -6cm \setlength{\unitlength}{0.16in}
\begin{picture}(24,4)
\put(15.25,2){\makebox(0,0)[c]{$\bigotimes$}}
\put(17.4,2){\makebox(0,0)[c]{$\bigcirc$}}
\put(21.9,2){\makebox(0,0)[c]{$\bigcirc$}}
\put(24.4,2){\makebox(0,0)[c]{$\bigotimes$}}
\put(26.8,2){\makebox(0,0)[c]{$\bigotimes$}}
\put(13.2,2){\line(1,0){1.45}} \put(15.7,2){\line(1,0){1.25}}
\put(17.8,2){\line(1,0){0.9}} \put(20.1,2){\line(1,0){1.4}}
\put(22.35,2){\line(1,0){1.6}} \put(24.9,2){\line(1,0){1.5}}
\put(27.3,2){\line(1,0){1.5}}
\put(19.5,1.95){\makebox(0,0)[c]{$\cdots$}}
\put(29.7,1.95){\makebox(0,0)[c]{$\cdots$}}
\put(15.2,3){\makebox(0,0)[c]{\tiny $-\sum_{i=1}^{2n-1}\alpha_{i/2}$}}
\put(17.8,1){\makebox(0,0)[c]{\tiny $\beta_{1/2}$}}
\put(22,1){\makebox(0,0)[c]{\tiny $\beta_{n-1/2}$}}
\put(24.4,1){\makebox(0,0)[c]{\tiny $\alpha_{n+1/2}$}}
\put(27,1){\makebox(0,0)[c]{\tiny $\alpha_{n+1}$}}
\put(8.0,2){\makebox(0,0)[c]{{\ovalBox(1.6,1.2){$\mf{k}$}}}}
\put(8.8,2){\line(1,0){1.7}}
\put(11.8,2){\makebox(0,0)[c]{{\ovalBox(2.6,1.2){$\mf{T}_{n}$}}}}
\end{picture}
\end{center}
The crucial point here is that the subdiagram to the left of
the first $\bigotimes$ is the Dynkin diagram of $\G_{n}$.

On the other hand, starting with the third Dynkin diagram in
\eqref{Dynkin:combined} associated to $\DG$, we apply the
following new sequence of $\frac{n(n+1)}{2}$ odd reflections.
First we apply one odd reflection corresponding to $\alpha_{1}$,
then we apply two odd reflections corresponding to $\alpha_{2}$
and $\alpha_{1}+\alpha_{3/2}+\alpha_2$.  After that we apply three
odd reflections corresponding to $\alpha_{3}$,
$\alpha_{2}+\alpha_{5/2}+\alpha_3$, and
$\alpha_1+\alpha_{3/2}+\alpha_{2}+\alpha_{5/2}+\alpha_3$, et
cetera, until finally we apply $n$ odd reflections corresponding
to $\alpha_{n},\alpha_{n-1}+\alpha_{n-1/2}
+\alpha_{n},\ldots,\sum_{i=2}^{2n}\alpha_{i/2}$. The resulting new
Borel subalgebra for $\DG$ will be denoted by $\wt{\mf{b}}^{s}(n)$
and its corresponding simple roots are listed in the following
Dynkin diagram:
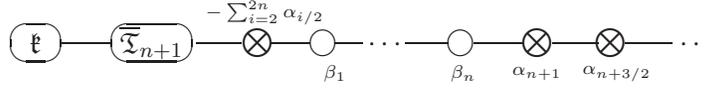
\begin{center}
\hskip -6cm \setlength{\unitlength}{0.16in}
\begin{picture}(24,4)
\put(15.25,2){\makebox(0,0)[c]{$\bigotimes$}}
\put(17.4,2){\makebox(0,0)[c]{$\bigcirc$}}
\put(21.9,2){\makebox(0,0)[c]{$\bigcirc$}}
\put(24.4,2){\makebox(0,0)[c]{$\bigotimes$}}
\put(26.8,2){\makebox(0,0)[c]{$\bigotimes$}}
\put(13.28,2){\line(1,0){1.45}}
\put(15.7,2){\line(1,0){1.25}} \put(17.8,2){\line(1,0){0.9}}
\put(20.1,2){\line(1,0){1.4}}
\put(22.35,2){\line(1,0){1.6}}
\put(24.9,2){\line(1,0){1.5}}
\put(27.3,2){\line(1,0){1.5}}
\put(19.5,1.95){\makebox(0,0)[c]{$\cdots$}}
\put(29.7,1.95){\makebox(0,0)[c]{$\cdots$}}
\put(15.5,3){\makebox(0,0)[c]{\tiny $-\sum_{i=2}^{2n}\alpha_{i/2}$}}
\put(17.8,1){\makebox(0,0)[c]{\tiny $\beta_{1}$}}
\put(22,1){\makebox(0,0)[c]{\tiny $\beta_{n}$}}
\put(24.4,1){\makebox(0,0)[c]{\tiny $\alpha_{n+1}$}}
\put(27,1){\makebox(0,0)[c]{\tiny $\alpha_{n+3/2}$}}
\put(8.0,2){\makebox(0,0)[c]{{\ovalBox(1.6,1.2){$\mf{k}$}}}}
\put(8.8,2){\line(1,0){1.7}}
\put(11.8,2){\makebox(0,0)[c]{{\ovalBox(2.6,1.2){$\ov{\mf{T}}_{n+1}$}}}}
\end{picture}
\end{center}
We remark that the subdiagram to the left of the odd simple
root $-\sum_{i=2}^{2n}\alpha_{i/2}$ above becomes the Dynkin
diagram of $\SG_{n+1}$.

%
%
\subsection{Highest weights with respect to $\wt{\mf{b}}^{c}(n)$
and $\wt{\mf{b}}^{s}(n)$}

Recall the standard Levi subalgebra $\wt{\mf{l}}$ of $\DG$ with (opposite) nilradical
$\wt{\mf{u}}$ and $\wt{\mf{u}}_-$ (see Section~\ref{sec:OBCD}).
\begin{lem}\label{lem:uinvar}
The sequences of odd reflections in \ref{oddreflection}
leave the sets of roots of $\wt{\mf{u}}$
and $\wt{\mf{u}}_-$ invariant.
\end{lem}

\begin{proof}
This follows from the fact that the simple roots used in the
sequences of odd reflections in Section~ \ref{oddreflection}
are all roots of $\wt{\mf{l}}$.
\end{proof}

We denote by $\wt{\mf{b}}^{c}_{\wt{\mf{l}}}(n)$ and
$\wt{\mf{b}}^{s}_{\wt{\mf{l}}}(n)$ the Borel subalgebras of
$\wt{\mf{l}}$ corresponding to the sets of simple roots
$\wt{\Pi}^{c}(n)\cap\sum_{\alpha\in\wt{Y}}\Z\alpha$ and
$\wt{\Pi}^{s}(n)\cap\sum_{\alpha\in\wt{Y}}\Z\alpha$, respectively.
The sequences of odd reflections in \ref{oddreflection} only affect the tail diagram
\makebox(20,0){$\oval(20,14)$}\makebox(-20,8){$\wt{\mf{T}}$} and
leaves the head diagram
\makebox(20,0){$\oval(20,12)$}\makebox(-20,7){$\mf{k}$} invariant.
Since the tail diagram is of type $A$, the proofs of \cite[Lemma
3.2]{CL} and \cite[Corollary 3.3]{CL} can be adapted in a
straightforward way to prove the following
(where Lemma~\ref{lem:uinvar} is used).

\begin{prop}\label{prop:change} Let $\la\in P^+$ and $n\in\N$.
\begin{itemize}
\item[(i)] Suppose that $\ell(\la_+)\le n$. Then the highest
weight of $L(\wt{\mf{l}},\la^\theta)$ with respect to the Borel
subalgebra $\wt{\mf{b}}^{c}_{\wt{\mf{l}}}(n)$ is $\la$.
Furthermore, $\wt{\Delta}(\la^\theta)$ and $\wt{L}(\la^\theta)$
are highest weight $\DG$-modules of highest weight $\la$ with
respect to the new Borel subalgebra $\wt{\mf{b}}^{c}(n)$.

\item[(ii)] Suppose that $\ell(\la_+')\le n$. Then the highest
weight of $L(\wt{\mf{l}},\la^\theta)$ with respect to the Borel
subalgebra $\wt{\mf{b}}^{s}_{\wt{\mf{l}}}(n)$ is $\la^\natural$.
Furthermore, $\wt{\Delta}(\la^\theta)$ and $\wt{L}(\la^\theta)$
are highest weight $\DG$-modules of highest weight $\la^\natural$
with respect to the new Borel subalgebra $\wt{\mf{b}}^{s}(n)$.
\end{itemize}
\end{prop}

\subsection{The functors $T$ and $\ov{T}$}\label{Tfunctors}
\label{sec:T}

By definition, $\G$ and $\SG$ are naturally subalgebras of $\DG$, $\mf l$ and
$\ov{\mf{l}}$ are subalgebras of $\wt{\mf{l}}$, while $\h$ and $\ov{\h}$ are subalgebras
of $\wt{\h}$. Also, we may regard ${\h}^* \subset \wt{\h}^*$ and $\ov{\h}^* \subset
\wt{\h}^*$.

Given a semisimple $\wt{\h}$-module
$\wt{M}=\bigoplus_{\gamma\in\wt{\h}^*}\wt{M}_\gamma$, we define
\begin{align*}
T(\wt{M}):= \bigoplus_{\gamma\in{{\h}^*}}\wt{M}_\gamma,\qquad
\hbox{and}\qquad \ov{T}(\wt{M}):=
\bigoplus_{\gamma\in{\ov{\h}^*}}\wt{M}_\gamma.
\end{align*}
Note that $T(\wt{M})$ is an ${\h}$-submodule of the $\wt{M}$, and
$\ov{T}(\wt{M})$ is an $\ov{\h}$-submodule of $\wt{M}$. One checks
that if $\wt{M}=\bigoplus_{\gamma\in\wt{\h}^*}\wt{M}_\gamma$ is an
$\wt{\mf{l}}$-module, then $T(\wt{M})$ is an ${\mf{l}}$-submodule
of $\wt{M}$ and $\ov{T}(\wt{M})$ is an $\ov{\mf{l}}$-submodule of
$\wt{M}$. Furthermore, if
$\wt{M}=\bigoplus_{\gamma\in\wt{\h}^*}\wt{M}_\gamma$ is a
$\DG$-module, then $T(\wt{M})$ is a ${\G}$-submodule of $\wt{M}$
and $\ov{T}(\wt{M})$ is a $\ov{\G}$-submodule of $\wt{M}$.

The direct sum decomposition in $\wt{M}$ gives rise to the natural
projections
\begin{eqnarray*}
\CD
 T_{\wt{M}} :  \wt{M} @>>>T(\wt{M}) \qquad
 \hbox{and}\qquad  \ov{T}_{\wt{M}} :  \wt{M} @>>>\ov{T}(\wt{M})
 \endCD
\end{eqnarray*}
that are ${\h}$- and $\ov{\h}$-module homomorphisms, respectively.
If $\wt{f}:\wt{M}\rightarrow \wt{N}$ is an
$\wt{\h}$-homomorphism, then the following induced maps
\begin{eqnarray*}
 \CD
T[\wt{f}] :  T(\wt{M}) @>>>T(\wt{N}) \qquad \hbox{and}\qquad
\ov{T}[\wt{f}] :  \ov{T}(\wt{M}) @>>>\ov{T}(\wt{N})
 \endCD
\end{eqnarray*}
are also ${\h}$- and $\ov{\h}$-module homomorphisms, respectively.
Also if $\wt{f}:\wt{M}\rightarrow \wt{N}$
is a ${\wt{\G}}$-homomorphism, then $T_{\wt{M}}$ and $T[\wt{f}]$
(respectively, $\ov{T}_{\wt{M}}$ and $\ov{T}[\wt{f}]$) are ${\G}$-
(respectively, $\ov{\G}$-) homomorphisms.

\begin{lem}\label{lmod2lmod}
For $\la\in P^+$, we have
$T\big{(}L(\wt{\mf{l}},\la^\theta)\big{)} = L({\mf{l}},\la)$, and
$\ov{T}\big{(}L(\wt{\mf{l}},\la^\theta)\big{)}=
L(\ov{\mf{l}},\la^\natural)$.
\end{lem}

\begin{proof}
We shall prove the first formula using a character argument,
and the second one can be proved similarly.

Associated to partitions $\nu \subset \la$, we denote by $s_\la(x_1,x_2,\ldots)$ and
$s_{\la/\nu}(x_1,x_2,\ldots)$ the Schur and skew Schur functions in the variables
$x_1,x_2,\ldots$. The hook Schur functions associated to $\la$ is defined to be
(cf.~\cite{S, BR})
\begin{equation}\label{hookschur:char}
hs_{\la}(x_{1/2},x_1,x_{3/2},x_2,\ldots)
:=\sum_{\mu\subset\la}s_{\mu}(x_{1/2},x_{3/2},\ldots)
s_{\la'/\mu'}(x_1,x_2,\ldots).
\end{equation}
For a dominant tuple $(\la_{-m},\ldots,\la_{-1};\la^+)$, we have (cf. \cite{CK})
\begin{equation} \label{eqn:hookschur}
{\rm ch}L(\wt{\mf{l}},\la^\theta)={\rm
ch}L(\wt{\mf{l}}\cap\mf{k},\la|_{\mf k}) \,
hs_{\la_+'}(x_{1/2},x_1,x_{3/2},x_2,\ldots).
\end{equation}
Here $x_r:=e^{\epsilon_r}$ for each $r$,
 and $L(\wt{\mf{l}}\cap\mf{k},\la|_{\mf k})$ denotes the irreducible
$\wt{\mf{l}}\cap\mf{k}$-module of highest weight $\la|_{\mf
k}=\sum_{i=-m}^{-1}\la_i\epsilon_i$. Note that
$\wt{\mf{l}}\cap\mf{k} ={\mf{l}}\cap\mf{k}$.

As an $\mf{l}$-module,
$L(\wt{\mf{l}},\la^\theta)$ is completely reducible.
On the character level,
applying $T$ to $L(\wt{\mf{l}},\la^\theta)$
corresponds to setting $x_{1/2},x_{3/2},x_{5/2},\ldots$
in the character formula \eqnref{eqn:hookschur} to zero.
By \eqnref{hookschur:char},
$T\big{(}L({\mf{l}},\la^\theta)\big{)}$ is an $\mf{l}$-module with
character ${\rm
ch}L(\wt{\mf{l}}\cap\mf{k},\la|_{\mf k}) \,
s_{\la_+}(x_1,x_2,\ldots)$, which is precisely the character of
$L({\mf{l}},\la)$. This proves the formula.
\end{proof}

\begin{cor}\label{functor}
$T$ and $\ov{T}$ define exact functors from $\wt{\mc{O}}$ to
$\mc{O}$ and from $\wt{\mc{O}}$ to $\ov{\mc{O}}$, respectively.
\end{cor}

The following theorem can be regarded as a weak version of the
super duality which is to be established in
Theorem~\ref{thm:equivalence}.

\begin{thm}\label{matching:modules}
Let $\la\in P^+$. If $\wt{M}$ is a highest weight $\wt{\G}$-module
of highest weight $\la^\theta$, then $T(\wt{M})$ and
$\ov{T}(\wt{M})$ are highest weight ${\G}$- and $\ov{\G}$-modules
of highest weights $\la$ and $\la^\natural$, respectively.
Furthermore, we have
\begin{align*}
T\big{(}\wt{\Delta}(\la^\theta)\big{)}
 =\Delta(\la),\quad &T\big{(}\wt{L}(\la^\theta)\big{)}=L(\la);
    \\
\ov{T}\big{(}\wt{\Delta}(\la^\theta)\big{)}
 =\ov{\Delta}(\la^\natural),\quad
&\ov{T}\big{(}\wt{L}(\la^\theta)\big{)}=\ov{L}(\la^\natural).
\end{align*}
\end{thm}
\begin{proof}
We will prove only the statements involving $T$, and the
statements involving $\ov{T}$ can be proved in the same way.

By \propref{prop:change}, $\wt{M}$ contains a $\wt{\mf
b}^c(n)$-highest weight vector $v_\la$ of highest weight $\la$ for
$n \gg 0$.  The vector $v_\la$ clearly lies in $T(\wt{M})$, and it
is a ${\mf b}$-singular vector since $\mf{b}=\G\cap\wt{\mf
b}^c(n)$. Now $T(\wt{M})$ is completely reducible over ${\mf l}$
with all highest weights of its irreducible summands lying in
$P^+$. Thus to show that $T(\wt{M})$ is a highest weight
$\G$-module it remains to show that any vector in $T(\wt{M})$ of
weight in $P^+$ is contained in ${U}(\G)v_\la$. This follows by
the same argument in \cite[Lemma 3.5]{CL}, which only relies on
the $A$-type tail diagram of $\DG$.

Let us write $\Delta(\la) =U(\mf u_-) \otimes_\C L({\mf{l}},\la)$ and
$\wt{\Delta}(\la^\theta) =U(\wt{\mf u}_-) \otimes_\C L(\wt{\mf{l}},\la^\theta)$. We
observe that all the weights in $U(\mf u_-)$, $L({\mf{l}},\la)$, $U(\wt{\mf u}_-)$, and
$L(\wt{\mf{l}},\la^\theta)$ are of the form $\sum_{j<0}a_j\ep_j+\sum_{r>0}b_r\ep_r$ with
$b_r\in\Z_+$. Since also $T(U(\wt{\mf u}_-)) =U(\mf u_-)$, it follows by
\lemref{lmod2lmod} that  $\text{ch}T\big{(}\wt{\Delta}(\la^\theta)\big{)}
=\text{ch}\Delta(\la)$. Since $T\big{(}\wt{\Delta}(\la^\theta)\big{)}$ is a highest
weight module of highest weight $\la$, we have $T\big{(}\wt{\Delta}(\la^\theta)\big{)}
=\Delta(\la)$.

To show that $T$ sends irreducibles to irreducibles we show that
$T(\wt{L}(\la^\theta))$  has no singular vector apart from the
scalar multiples of a highest weight vector.  We argue by assuming
otherwise and derive a contradiction.  If we have another singular
vector of weight different from $\la$, then we can show, following
the second part of the proof of \cite[Theorem 3.6]{CL}, that we
also have a singular vector in $\wt{L}(\la^\theta)$ of weight
different from $\la^\theta$. The argument there is applicable here,
since it again only depends on the tail diagram, which is
of type $A$.
\end{proof}

\begin{rem}
It can be shown that tilting modules exist in categories $\mc O, \ov{\mc O}, \wt{\mc O}$
(cf. \cite{CWZ, CW2} for type $A$) and that the functors $T$ and $\ov T$ respect the
tilting modules. We choose not to develop the details in order to keep the paper to a
reasonable length.
\end{rem}

By standard arguments \thmref{matching:modules} implies the
following character formula.

\begin{thm}\label{character}
Let $\la\in P^+$, and write ${\rm ch}L(\la)=\sum_{\mu\in
P^+}a_{\mu\la}{\rm ch}\Delta(\mu)$, $a_{\mu\la}\in\Z$.  Then
\begin{itemize}
\item[(i)] ${\rm ch}\ov{L}(\la^\natural)=\sum_{\mu\in
P^+}a_{\mu\la}{\rm ch}\ov{\Delta}(\mu^\natural)$,

\item[(ii)] ${\rm ch}\wt{L}(\la^\theta)=\sum_{\mu\in
P^+}a_{\mu\la}{\rm ch}\wt{\Delta}(\mu^\theta)$.
\end{itemize}
\end{thm}

\begin{rem} \label{rem:sol}
The transition matrix $(a_{\mu\la})$ in \thmref{character} is known according to the
Kazhdan-Lusztig theory. This is because the Kazhdan-Lusztig polynomials in the BGG
category ${O}$ also determine the composition factors of generalized Verma modules in the
corresponding parabolic subcategory (see e.g.~\cite[p.~445 and Proposition 7.5]{So}). Hence
\thmref{character} and \lemref{lem:trunc} provide a complete solution to the irreducible
character problem in the category $\ov{\mc O}_n$ for the ortho-symplectic Lie
superalgebras.
\end{rem}

\subsection{Kostant type homology formula}\label{sec:homology}

For a precise definition of homology groups of Lie superalgebras
with coefficients in a module and a precise formula for the
boundary operator we refer the reader to \cite[Section 4]{CL} or
\cite{T}.

For $\wt{M}\in\wt{\mc{O}}$ we denote by $M=T(\wt{M})\in{\mc{O}}$
and $\ov{M}=\ov{T}(\wt{M})\in\ov{\mc{O}}$. Furthermore let
$\wt{d}:\Lambda(\wt{\mf{u}}_-)\otimes
{\wt{M}}\rightarrow\Lambda(\wt{\mf{u}}_-)\otimes {\wt{M}}$ be the
boundary operator of the complex of $\wt{\mf{u}}_-$-homology
groups with coefficients in $\wt{M}$,
regarded as a $\wt{\mf{u}}_-$-module. The map $\wt{d}$ is an
$\wt{\mf{l}}$-module homomorphism and hence the homology groups
$H_n(\wt{\mf{u}}_-,\wt{M})$ are $\wt{\mf{l}}$-modules, for
$n\in\Z_+$.
Accordingly we let $d:\Lambda({\mf{u}}_-)\otimes {M}\rightarrow
\Lambda({\mf{u}}_-)\otimes {M}$ and
$\ov{d}:\Lambda({\ov{\mf{u}}}_-)\otimes \ov{M}\rightarrow
\Lambda({\ov{\mf{u}}}_-)\otimes \ov{M}$ stand for the boundary
operator of the complex of $\mf{u}_-$-homology with coefficients
in $M$ and the boundary operator of the complex of
${\ov{\mf{u}}}_-$-homology with coefficients in $\ov{M}$,
respectively. Similarly, ${d}$ and $\ov{d}$ are ${\mf{l}}$- and
${\ov{\mf{l}}}$-homomorphisms, respectively.

\begin{lem}\label{boundary}
For $\wt{M}\in{\wt{\mc{O}}}$ and $\la\in P^+$, we have
\begin{itemize}\label{lem:aux3}
\item[(i)] $T\big{(}\La(\wt{\mf{u}}_-)
\otimes\wt{M}\big{)}=\La({\mf{u}}_-)\otimes{M}$,
 and thus
 $T\big{(}\La(\wt{\mf{u}}_-)\otimes\wt{L}(\la^\theta)\big{)}
 = \La({\mf{u}}_-)\otimes L(\la).$
 Moreover, $T[\wt{d}]=d$.

\item[(ii)] $\ov{T}\big{(}\La(\wt{\mf{u}}_-)\otimes\wt{M}\big{)}
=\La({\ov{\mf{u}}}_-)\otimes{\ov{M}}$,
 and thus
 $\ov{T}\big{(}\La(\wt{\mf{u}}_-)\otimes
\wt{L}(\la^\theta)\big{)} =
\La({\ov{\mf{u}}}_-)\otimes\ov{L}(\la^\natural).$ Moreover,
$\ov{T}[\wt{d}]=\ov{d}$.
\end{itemize}
\end{lem}

\begin{proof}
We will prove (i) only. It follows by definition of $T$ and $\wt{\mf{u}}_-$ that
$T\big{(}\Lambda(\wt{\mf{u}}_-)\big{)}=\Lambda({\mf{u}}_-)$. Now, since all modules
involved have weights of the form $\sum_{i<0}a_i\ep_i+\sum_{r>0}b_r\ep_r$ with
$b_r\in\Z_+$, it follows that $T\big{(}\La(\wt{\mf{u}}_-)\otimes\wt{M}\big{)}$ and
$\La({\mf{u}}_-)\otimes{M}$ have the same character. Complete reducibility of the
$\mf{l}$-modules $T\big{(}\La(\wt{\mf{u}}_-)\otimes\wt{M}\big{)}$ and
$\La({\mf{u}}_-)\otimes{M}$ implies that
$T\big{(}\La(\wt{\mf{u}}_-)\otimes\wt{M}\big{)}=\La({\mf{u}}_-)\otimes{M}$ as
$\mf{l}$-modules. \thmref{matching:modules} completes the proof of the first part of (i).

The second part of (i) follows from the definitions of $\wt{d}$ and $d$ (see
e.g.~\cite[(4.1)]{CL}).
\end{proof}

By \lemref{boundary} we have the following commutative diagram.
\begin{eqnarray}\label{compare-complexes}
\CD
\cdots @>\wt{d}>> \La^{n+1}(\wt{\mf{u}}_-)
 \otimes\wt{M} @>\wt{d}>>\La^{n}(\wt{\mf{u}}_-)
 \otimes\wt{M} @>\wt{d}>>\La^{n-1}(\wt{\mf{u}}_-)\otimes\wt{M}\cdots \\
@.  @VVT_{\La^{n+1}(\wt{\mf{u}}_-)
 \otimes\wt{M}}V @VVT_{\La^{n}(\wt{\mf{u}}_-)
 \otimes\wt{M}}V @VVT_{\La^{n-1}(\wt{\mf{u}}_-)\otimes\wt{M}}V\\
\cdots @>d>> \La^{n+1}({\mf{u}}_-)
 \otimes {M} @>{d}>>\La^{n}({\mf{u}}_-)
 \otimes{M} @>{d}>>\La^{n-1}({\mf{u}}_-)\otimes{M}\cdots \\
 \endCD
\end{eqnarray}
Thus $T$ induces an $\mf{l}$-homomorphism from
$H_n(\wt{\mf{u}}_-;\wt{M})$ to $H_n({\mf{u}}_-;{M})$. Similarly,
$\ov{T}$ induces an ${\ov{\mf{l}}}$-homomorphism from
$H_n(\wt{\mf{u}}_-;\wt{M})$ to $H_n({\ov{\mf{u}}}_-;\ov{M})$.

As an ${\wt{\mf{l}}}$-module, $\La(\wt{\mf{u}}_-)$ is a direct sum
of $L({\wt{\mf{l}}},\mu^\theta)$, $\mu\in P^+$, each appearing
with finite multiplicity (\cite[Section 3.2.3]{CK}). By
\cite[Theorem~3.2]{CK}, $\Lambda(\wt{\mf{u}}_-)\otimes\wt{M}$ as
an ${\wt{\mf{l}}}$-module is completely reducible. Write
$\Lambda(\wt{\mf{u}}_-)\otimes\wt{M}\cong\bigoplus_{\mu\in
P^+}L({\wt{\mf{l}}},\mu^\theta)^{m(\mu)}$ as
${\wt{\mf{l}}}$-modules.  It follows by Lemmas \ref{lmod2lmod} and
\ref{lem:aux3} that $\Lambda({\mf{u}}_-)\otimes{M}\cong
\bigoplus_{\mu\in P^+}L({\mf{l}},\mu)^{m(\mu)}$, as
${\mf{l}}$-modules. Similarly,
$\Lambda({\ov{\mf{u}}}_-)\otimes\ov{M}\cong\bigoplus_{\mu\in
P^+}L({\ov{\mf{l}}},\mu^\natural)^{m(\mu)}$, as
${\ov{\mf{l}}}$-modules. The commutativity of
\eqnref{compare-complexes} and \lemref{boundary} now allow us to
adapt the proof of \cite[Theorem ~4.4]{CL} to prove the following.

\begin{thm}\label{matching:KL} We have for $n\ge 0$
\begin{itemize}
\item[(i)] $T(H_n(\wt{\mf{u}}_-;\wt{M}))\cong
H_n({\mf{u}}_-;{M})$, as $\mf{l}$-modules.

\item[(ii)] $\ov{T}(H_n(\wt{\mf{u}}_-;\wt{M}))\cong
H_n({\ov{\mf{u}}}_-;\ov{M})$, as ${\ov{\mf{l}}}$-modules.
\end{itemize}
\end{thm}

Setting $\wt{M}=\wt{L}(\la^\theta)$ in \thmref{matching:KL} and using
\thmref{matching:modules} we obtain the following.

\begin{cor}\label{matching:KL1}
For $\la\in P^+$ and $n\ge 0$, we have
\begin{itemize}
\item[(i)]
$T\big{(}H_n(\wt{\mf{u}}_-;\wt{L}(\la^\theta))\big{)}\cong
H_n({\mf{u}}_-;L(\la))$, as $\mf{l}$-modules.

\item[(ii)]
$\ov{T}\big{(}H_n(\wt{\mf{u}}_-;\wt{L}(\la^\theta))\big{)}\cong
H_n({\ov{\mf{u}}}_-;\ov{L}(\la^\natural))$, as
${\ov{\mf{l}}}$-modules.
\end{itemize}
\end{cor}

We define parabolic Kazhdan-Lusztig polynomials in the categories
$\mc O$, $\ov{\mc O}$ and $\wt{\mc O}$ for $\mu,\la\in P^+$ by
letting
\begin{align*}
{\ell}_{\mu\la}(q) :=\sum_{n=0}^\infty\text{dim}_\C
\Big{(}\text{Hom}_{{\mf{l}}}\big{[} L({\mf{l}},\mu),
H_n\big{(}{{\mf{u}}}_-;{L}(\la)\big{)} \big{]}\Big{)} (-q)^{-n},\\
\ov{\ell}_{\mu^\natural\la^\natural}(q)
 :=\sum_{n=0}^\infty\text{dim}_\C\Big{(}\text{Hom}_{\ov{\mf{l}}}\big{[}
 L(\ov{\mf{l}},\mu^\natural),
 H_n\big{(}{\ov{\mf{u}}}_-;{L}(\la^\natural)\big{)} \big{]}\Big{)} (-q)^{-n},\\
\wt{\ell}_{\mu^\theta\la^\theta}(q)
 :=\sum_{n=0}^\infty\text{dim}_\C\Big{(}\text{Hom}_{\wt{\mf{l}}}\big{[}
 L(\wt{\mf{l}},\mu^\theta),
 H_n\big{(}{\wt{\mf{u}}}_-;{L}(\la^\theta)\big{)} \big{]}\Big{)} (-q)^{-n}.
\end{align*}

By Vogan's homological interpretation of the
Kazhdan-Lusztig polynomials \cite[Conjecture 3.4]{V} and the
Kazhdan-Lusztig conjectures \cite{KL}, proved in \cite{BB, BK},
$\ell_{\mu\la}(q)$ coincides with the original definition
and moreover $\ell_{\mu\la} (1) =a_{\mu\la}$ (cf. \thmref{character}).
The following reformulation of \corref{matching:KL1}
 is a generalization of \thmref{character}.

\begin{thm}   \label{matching:KLpol}
For $\la,\mu\in P^+$ we have $ \ell_{\mu\la}(q)
=\wt{\ell}_{\mu^\theta\la^\theta}(q) =
\ov{\ell}_{\mu^\natural\la^\natural}(q). $
\end{thm}

\section{Equivalences of categories}\label{sec:category}
\subsection{}

The following is standard (see, for example, \cite[Lemma 2.1.10]{Ku}).

\begin{prop}\label{Ku}
Let ${M}\in {\mc{O}}$. Then there exists a (possibly infinite) increasing filtration
$0={M}_0\subset {M}_1\subset{M}_2\subset\cdots$ of ${\G}$-modules such that
\begin{itemize}
\item[(i)] $\bigcup_{i\ge 0}{M}_i={M}$,

\item[(ii)] ${M}_i/{M}_{i-1}$ is a highest weight module of highest weight $\nu_i$
    with $\nu_i\in{P^+}$, for $i\ge 1$,

\item[(iii)] the condition $\nu_i-\nu_j\in\sum_{\alpha\in{\Pi}}\Z_+\alpha$ implies
    that $i<j$,

\item[(iv)] for any weight $\mu$ of ${M}$, there exists an $r\in \N$ such that
    $({M}/{M}_{r})_\mu=0$.
\end{itemize}

Similar statements hold for $\ov{M}\in {\ov{\mc{O}}}$ and $\wt{M}\in \wt{{\mc{O}}}$.
\end{prop}

Let $\wt{\mc{O}}^f$ denote the full subcategory of $\wt{\mc{O}}$ consisting of finitely
generated ${U}(\wt{\G})$-modules. The categories $\ov{\mc{O}}^f$ and $\mc{O}^f$ are
defined in a similar fashion.

\propref{Ku} implies the following.

\begin{prop}\label{filtration}
Let ${M}\in {{\mc{O}}}$. Then ${M}\in {{\mc{O}^f}}$ if and only if there exists a finite
increasing filtration $0={M}_0\subset {M}_1\subset{M}_2\subset\cdots\subset{M}_k={M}$ of
${\G}$-modules such that ${M}_i/{M}_{i-1}$ is a highest weight module of highest weight
$\nu_i$ with $\nu_i\in{{P^+}}$, for $1\le i\le k$. Similar statements hold for $\ov{M}\in
{\ov{\mc{O}}}$ and $\wt{M}\in \wt{{\mc{O}}}$.
\end{prop}

The following proposition is the converse to \thmref{matching:modules}.
\begin{prop}\label{hw-onto}
\begin{itemize}
\item[(i)] If ${V}(\la)$ is a highest weight ${\G}$-module of highest weight
    $\la\in{P^+}$, then there is a highest weight ${\wt{\G}}$-module
    $\wt{V}(\la^\theta)$ of highest weight $\la^\theta$ such that
    $T(\wt{V}(\la^\theta))={V}(\la)$.

\item[(ii)] If $\ov{V}(\la^\natural)$ is a highest weight $\ov{\G}$-module of highest
    weight $\la^\natural$ with $\la\in P^+$, then there is a ${\wt{\G}}$-module
    $\wt{V}(\la^\theta)$ of highest weight $\la^\theta$ such that
    $\ov{T}(\wt{V}(\la^\theta))=\ov{V}(\la^\natural)$.
\end{itemize}
\end{prop}

\begin{proof}
We shall only prove (i), as (ii) is similar.  We let $W$ be the kernel of the natural
projection from the $\Delta(\la)$ to ${V}(\la)$. Now \thmref{matching:modules} says that
$T(\wt{\Delta}(\la^\theta))=\Delta(\la)$. Thus, by the exactness of functor $T$, it
suffices to prove that $W$ lifts to a submodule $\wt{W}$ of $\wt{\Delta}(\la^\theta)$
such that $T(\wt{W})=W$.

There is an increasing filtration $0=W_0\subset W_1\subset W_2\subset\cdots$ of
$\G$-modules for $W$ satisfying the properties of \propref{Ku}. For each $i>0$, let $v_i$
be a weight vector in $W_i$ such that $v_i+W_{i-1}$ is a non-zero highest weight vector
of $W_i/W_{i-1}$. Observe that $\wt{\Delta}(\la^\theta)=\bigoplus_{\mu\in
P^+}L({\wt{\mf{l}}},\mu^\theta)^{m(\mu)}$ and ${\Delta}(\la)=\bigoplus_{\mu\in
P^+}L({{\mf{l}}},\mu)^{m(\mu)}$ are completely reducible $\wt{\mf{l}}$-  and
$\mf{l}$-modules, respectively. Then, for each $i>0$, there is a highest weight vector
$\wt{v}_i$ of the $\wt{\mf{l}}$-module $U(\wt{\mf{l}})v_i$ with respect to the Borel
subalgebra $\wt{\mf{b}}\cap\wt{\mf{l}}$. Let $\wt{W}_i$ be the submodule of
$\wt{\Delta}(\la^\theta)$ generated by $\wt{v}_1, \wt{v}_2,\ldots, \wt{v}_i$ and set
$\wt{W}_0=0$. It is easy to see $\wt{v}_i$ is a highest weight vector of the
$\wt{\G}$-module $\wt{W}_i/\wt{W}_{i-1}$. Let $\wt{W}=\bigcup_{i\ge 1}\wt{W}_i$. It is
clear that $T(\wt{W}_i/\wt{W}_{i-1})\cong {W}_i/{W}_{i-1}$ for all $i$. This implies
$T(\wt{W}_i)=W_i$ for all $i$ and hence $T(\wt{W})=W$.
\end{proof}
\subsection{The categories $\wt{\mc{O}}^{f,\bar{0}}$ and
$\ov{\mc{O}}^{f,\bar{0}}$} \label{sec:51}

Define an equivalence relation $\sim$ on $\wt{\h}^*$ by letting $\mu \sim \nu$ if and
only if $\mu -\nu$ lies in the root lattice $\Z \wt \Phi$ of $\DG$. For each such
equivalence class $[\mu]$, fix a representative $[\mu]^o \in \wt{\h}^*$ and declare
$[\mu]^o$ to have $\Z_2$-grading $\bar 0$. For $\epsilon=\bar{0},\bar{1}$, set
(cf.~\cite[\S4-e]{B} and \cite[Section~2.5]{CL} for type $A$)
\begin{eqnarray*}
 {\wt{\h}^*}_\epsilon &=&
 \Big \{\mu\in
\wt{\h}^*\mid \sum_{r\in {1/ 2}+\Z_+}(\mu-[\mu]^o)
(\widehat{E}_{r})\equiv {\epsilon} \,\,(\text{mod }2) \Big \},
\text{ for } \xx =\mf{b, c, d},
 \\
 {\wt{\h}^*}_\epsilon &=&
 \Big \{\mu\in
\wt{\h}^*\mid \sum_{i=1}^m (\mu-[\mu]^o)(\widehat{E}_{-i}) +
\sum_{r\in \N}(\mu-[\mu]^o)(\widehat{E}_{r})\equiv {\epsilon}
\,\,(\text{mod }2) \Big \}, \text{ for } \xx =\mf b^\bullet.
\end{eqnarray*}
Recall that $\wt{V} \in \wt{\mc O}$ is a semisimple $\wt{\h}$-module with
$\wt{V}=\bigoplus_{\gamma\in\wt{\h}^*}\wt{V}_\gamma$. Then $\wt{V}$ acquires a natural
$\Z_2$-grading $\wt{V}=\wt{V}_{\bar{0}}\bigoplus \wt{V}_{\bar{1}}$ given by
 \begin{equation}\label{wt-Z2-gradation}
 \wt{V}_{\ep}
  :=\bigoplus_{\mu\in{\wt{\h}^*}_{\ep}}\wt{V}_{\mu}, \qquad \ep
  =\bar{0},\bar{1},
 \end{equation}
 which is compatible with the $\Z_2$-grading on $\DG$.

We define $\wt{\mc{O}}^{\bar{0}}$ and $\wt{\mc{O}}^{f,\bar{0}}$ to be the full
subcategories of $\wt{\mc{O}}$ and $\wt{\mc{O}}^f$, respectively, consisting of objects
with $\Z_2$-gradation given by \eqnref{wt-Z2-gradation}. Note that the morphisms in
$\wt{\mc{O}}^{\bar{0}}$ and $\wt{\mc{O}}^{f,\bar{0}}$ are of degree $\bar{0}$. For
$\wt{M}\in\wt{\mc{O}}$, let $\widehat{\wt{M}}\in\wt{\mc{O}}^{\bar{0}}$ denote the
$\wt{\G}$-module $\wt{M}$ equipped with the $\Z_2$-gradation given by
\eqnref{wt-Z2-gradation}. It is clear that $\widehat{\wt{M}}$ is isomorphic to $\wt{M}$
in $\wt{\mc{O}}$. Thus $\wt{\mc{O}}$ and $\wt{\mc{O}}^{\bar{0}}$ have isomorphic
skeletons and hence they are equivalent categories.  Similarly, $\wt{\mc{O}}^f$ and
$\wt{\mc{O}}^{f,\bar{0}}$ are equivalent categories.

Analogously define ${\mc{O}}^{\bar{0}}$, ${\mc{O}}^{f,\bar{0}}$, $\ov{\mc{O}}^{\bar{0}}$
and $\ov{\mc{O}}^{f,\bar{0}}$ to be the respective full subcategories of ${\mc{O}}$,
${\mc{O}}^f$, $\ov{\mc{O}}$ and $\ov{\mc{O}}^f$ consisting of objects with
$\Z_2$-gradation given by \eqnref{wt-Z2-gradation}. Similarly,
${\mc{O}}^{\bar{0}}\cong{\mc{O}}$, $\ov{\mc{O}}^{\bar{0}}\cong\ov{\mc{O}}$, and also
${\mc{O}}^{f,\bar{0}}\cong{\mc{O}}^f$, $\ov{\mc{O}}^{f,\bar{0}}\cong\ov{\mc{O}}^f$. (In
case of $\mc O$ and $\mc O^f$, these remarks are trivial except for the type $\mf
b^\bullet$ which corresponds to a Lie superalgebra).

\subsection{Equivalence of the categories}

Recall the functors $T$ and $\ov{T}$ from \secref{Tfunctors}. The following is the main
result of this section.

\begin{thm}\label{thm:equivalence}
\begin{itemize}
\item[(i)] $T:\wt{\mc{O}}\rightarrow{\mc{O}}$ is an equivalence of categories.

\item[(ii)] $\ov{T}:\wt{\mc{O}}\rightarrow\ov{\mc{O}}$ is an equivalence of
    categories.
\end{itemize}
Hence, the categories $\mc{O}$ and $\ov{\mc{O}}$ are equivalent.
\end{thm}

Since  $\ov{\mc{O}}^{\bar{0}}\cong \ov{\mc{O}}$ and $\wt{\mc{O}}^{\bar{0}}\cong
\wt{\mc{O}}$ it is enough to prove \thmref{thm:equivalence} for $\ov{\mc{O}}^{\bar{0}}$
and $\wt{\mc{O}}^{\bar{0}}$. In order to keep notation simple we will from now on drop
the superscript $\bar{0}$ and use $\ov{\mc{O}}$, $\wt{\mc{O}}$, $\ov{\mc{O}}^f$ and
$\wt{\mc{O}}^f$ to denote the respective categories $\ov{\mc{O}}^{\bar{0}}$,
$\wt{\mc{O}}^{\bar{0}}$, $\ov{\mc{O}}^{f,\bar{0}}$ and $\wt{\mc{O}}^{f,\bar{0}}$ for the
remainder of Section \ref{sec:category}.  Henceforth, when we write
$\wt{\Delta}(\la^\theta), \wt{L}(\la^\theta)\in\wt{\mc{O}}^f$, $\la\in P^+$, we will mean
the corresponding modules equipped with the $\Z_2$-gradation \eqnref{wt-Z2-gradation}.
Similar convention applies to $\ov{\Delta}(\la^\natural)$ and $\ov{L}(\la^\natural)$.

For ${M},{N}\in{\mc{O}}$ and $i\in \N$ the $i$th extension group ${\rm
Ext}_{{\mc{O}}}^i({M},{N})$ can be understood in the sense of Baer-Yoneda (see
e.g.~\cite[Chapter VII]{M}) and ${\rm Ext}_{{\mc{O}}}^0({M},{N}):={\rm
Hom}_{{\mc{O}}}({M},{N})$. In a similar way extensions in $\ov{\mc{O}}$ and $\wt{\mc{O}}$
are interpreted. From this viewpoint the exact functors $T$ and $\ov{T}$ induce natural
maps on extensions by taking the projection of the corresponding exact sequences.

\begin{thm}\label{thm:equivalence-f.g.} We have the following.
\begin{itemize}
\item[(i)] $T:\wt{\mc{O}}^f\rightarrow{\mc{O}}^f$ is an equivalence of categories.
\item[(ii)] $\ov{T}:\wt{\mc{O}}^f\rightarrow\ov{\mc{O}}^f$ is an equivalence of
    categories.
\item[(iii)] The categories ${\mc{O}}^f$ and $\ov{\mc{O}}^f$ are equivalent.
\end{itemize}
\end{thm}

\thmref{thm:equivalence-f.g.} can be proved following a similar strategy as the one used
to prove \cite[Theorem 5.1]{CL}. To avoid repeating similar arguments at great length we
will just point out the main differences between their proofs. In \cite[Section 5]{CL}
the main point is to prove that the functor $T$ induces isomorphisms $\text{Hom}_{\wt{\mc
O}}(\wt{M},\wt{N})\cong \text{Hom}_{{\mc O}}({M},{N})$ and $\text{Ext}^1_{\wt{\mc
O}}(\wt{L},\wt{N})\cong \text{Ext}^1_{{\mc O}}({L},{N})$, for $\wt{L}$ irreducible, and
$\wt{M},\wt{N}$ having finite composition series. From this \cite[Theorem 5.1]{CL} is
derived easily. As the isomorphism of the $\text{Hom}$ spaces imply the isomorphism of
the $\text{Ext}^1$ spaces \cite[Lemma 5.12]{CL}, we are reduced to establish the
isomorphism of the $\text{Hom}$ spaces. To prove the isomorphism of the $\text{Hom}$
spaces therein, the idea is to prove this isomorphism first for $\wt{M}$ irreducible, and
then to use induction on the length of the composition series of $\wt{M}$ to establish
the general case.

Now, thanks to \propref{hw-onto}, we can proceed similarly as in \cite[Section 5]{CL} to
prove \thmref{thm:equivalence-f.g.}. For this purpose we replace $\wt{L}$ by a highest
weight module, and $\wt{M}$ and $\wt{N}$ by finitely generated modules. As finitely
generated modules possess finite filtrations whose subquotients are highest weight
modules (cf. \propref{filtration}), we can now borrow the same type of induction
arguments from \cite{CL}, now inducting on the length of such a filtration instead of the
length of a composition series. Therefore, the proof of the isomorphisms is again reduced
to a special case, namely when $\wt{M}$ is a highest weight module. This case can then be
proved using similar arguments as the ones given in the proof of \cite[Lemmas~5.8]{CL}.
The case of $\ov{T}$ is completely analogous.

Having \thmref{thm:equivalence-f.g.} at our disposal we can now prove
\thmref{thm:equivalence}.

\begin{proof}[Proof of \thmref{thm:equivalence}]
Since the proofs of  (i) and (ii) are similar, we shall only prove (i). (iii) follows
from (i) and (ii). For  every $M\in\mc{O}$, there is an increasing filtration
$0=M_0\subset M_1\subset M_2\subset\cdots$ of $\G$-modules for $M$ with
$M_i\in{\mc{O}}^f$ satisfying the properties of \propref{Ku}. The filtration $\{M_i\}$ of
$M$ lifts to a filtration $\{\wt{M}_i\}$ with $\wt{M}_i \in\wt{\mc{O}}^f$ such that
$T(\wt{M}_i)\cong M_i$ by \thmref{thm:equivalence-f.g.}. It is clear that we have
$\wt{M}:=\bigcup_{i\ge 0}\wt{M}_i\in\wt{\mc{O}}$ and $T(\wt{M})\cong M$.

It is well known that a full and faithful functor $F: \mc{C}\mapsto \mc{C'}$, satisfying
the property that for every $M'\in \mc{C}'$ there exists $M\in \mc{C}$ with $F(M)\cong
M'$, is an equivalence of categories (see e.g.~\cite[Proposition 1.5.2]{P}).

Therefore it remains to show that $T$ is full and faithful. By
\propref{Ku}, for $\wt{M} \in\wt{\mc{O}}$, we may choose an
increasing filtration of $\DG$-modules $0=\wt{M}_0\subset
\wt{M}_1\subset\wt{M}_2\subset\cdots$ such that $\bigcup_{i\ge
0}\wt{M}_i=\wt{M}$ and $\wt{M}_i/\wt{M}_{i-1}$ is a highest weight
module of highest weight $\nu^\theta_i$ with $\nu_i\in{P^+}$, for
$i\ge 1$. Then the direct limit of $\wt{M}_i$ is
$\underrightarrow{\lim}\,\wt{M}_i \cong \wt{M}$ and ${\rm
Hom}_{\wt{\mc{O}}}(\wt{M},\wt{N})\cong\underleftarrow{\lim}\, {\rm
Hom}_{\wt{\mc{O}}}(\wt{M}_i,\wt{N})$ for every $\wt{N}
\in\wt{\mc{O}}$. Similarly we have $\underrightarrow{\lim}\,{M}_i
\cong {M}$ and ${\rm
Hom}_{{\mc{O}}}({M},{N})\cong\underleftarrow{\lim}\,{\rm
Hom}_{{\mc{O}}}({M}_i,{N})$ for $N=T(\wt{N})$. Furthermore, we
have the following commutative diagram (where $\varphi
=\underleftarrow{\lim} T_{\wt{M}_i,\wt{N}}$):
\begin{eqnarray*}
\CD   {\rm
Hom}_{\wt{\mc{O}}}\big{(}\wt{M},\wt{N}\big{)}
@>\cong>>\underleftarrow{\lim}\,
{\rm Hom}_{\wt{\mc{O}}}\big{(}\wt{M}_i,\wt{N}\big{)} \\
 @VVT_{\wt{M},\wt{N}}V @VV\varphi V \\
 {\rm Hom}_{{\mc{O}}}\big{(}{M},N\big{)}
  @>\cong>>\underleftarrow{\lim}\,{\rm Hom}_{{\mc{O}}}\big{(}{M}_i,N\big{)}\\
\endCD
\end{eqnarray*}
Using a similar argument as the one given in \cite[Lemma
5.10]{CL}, where we replace the induction on the length of
composition series therein by induction on the length of finite
increasing filtration $0=\wt{M}_0\subset
\wt{M}_1\subset\wt{M}_2\subset\cdots\subset\wt{M}_i$, we show that
$T_{\wt{M}_i,\wt{N}}:{\rm
Hom}_{\wt{\mc{O}}}(\wt{M}_i,\wt{N})\rightarrow {\rm
Hom}_{{\mc{O}}}({M}_i,N)$ are isomorphisms for each $i$. Therefore
$\varphi$ is an isomorphism and hence $T_{\wt{M},\wt{N}}$ is an
isomorphism. This completes the proof.
\end{proof}

\section{Finite dimensional representations}\label{finite:dim:repn}

The main purpose of this section is to determine the extremal weights of finite
dimensional irreducible modules over the ortho-symplectic Lie superalgebras with integral
highest weights. It follows that all such finite dimensional irreducible modules for the
ortho-symplectic Lie superalgebras are in the category $\ov{\mc O}_n$. We note that the
finite dimensional irreducible modules of non-integral highest weights are typical and so
their characters are known \cite[Theorem 1]{K2}.

\subsection{Extremal weights for $\osp(2m+1|2n)$}
\label{sec:type B}

Let us denote the weights of the natural $\osp(2m+1|2n)$-module $\C^{2n|2m+1}$ by $\pm
\delta_i, 0, \pm \vep_j$ for $1\le i\le n, 1\le j \le m$. We call a weight {\em
integral}, if it lies in $\Z$-span of the $\delta_i$'s and $\ov{\ep}_j$'s. The {\em
standard} Borel subalgebra $\mc B^{\text{st}}$ of $\osp(2m+1|2n)$ is the one associated
to the following set of simple roots
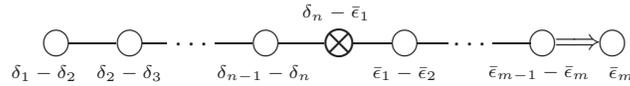
\begin{center}
\hskip -3cm \setlength{\unitlength}{0.16in}
\begin{picture}(24,4)
\put(6,2){\makebox(0,0)[c]{$\bigcirc$}}
\put(8.4,2){\makebox(0,0)[c]{$\bigcirc$}}
\put(10.5,1.95){\makebox(0,0)[c]{$\cdots$}}
\put(12.85,2){\makebox(0,0)[c]{$\bigcirc$}}
\put(15.25,2){\makebox(0,0)[c]{$\bigotimes$}}
\put(17.4,2){\makebox(0,0)[c]{$\bigcirc$}}
\put(19.6,1.95){\makebox(0,0)[c]{$\cdots$}}
\put(21.9,2){\makebox(0,0)[c]{$\bigcirc$}}
\put(24.2,2){\makebox(0,0)[c]{$\bigcirc$}}
\put(6.4,2){\line(1,0){1.55}} \put(8.82,2){\line(1,0){0.8}}
\put(11.2,2){\line(1,0){1.2}} \put(13.28,2){\line(1,0){1.45}}
\put(15.7,2){\line(1,0){1.25}} \put(17.8,2){\line(1,0){1.0}}
\put(20.3,2){\line(1,0){1.2}} \put(22.3,1.8){$\Longrightarrow$}
\put(5.6,1){\makebox(0,0)[c]{\tiny $\delta_1-\delta_2$}}
\put(8.4,1){\makebox(0,0)[c]{\tiny $\delta_2-\delta_3$}}
\put(12.8,1){\makebox(0,0)[c]{\tiny $\delta_{n-1}-\delta_n$}}
\put(15.15,3){\makebox(0,0)[c]{\tiny $\delta_n-\vep_{1}$}}
\put(17.4,1){\makebox(0,0)[c]{\tiny $\vep_{1}-\vep_{2}$}}
\put(21.8,1){\makebox(0,0)[c]{\tiny $\vep_{m-1} -\vep_{m}$}}
\put(24.5,1){\makebox(0,0)[c]{\tiny $\vep_{m}$}}
\end{picture}
\end{center}
An arbitrary Dynkin diagram for $\osp(2m+1|2n)$ always has a type
$A$ end while the other end is a short (even or odd) root.
Starting from the type $A$ end, the simple roots for a Borel
subalgebra $\mc B$ of $\osp(2m+1|2n)$ give rise to a sequence of
$d_1$ $\delta$'s, $e_1$ $\vep$'s, $d_2$ $\delta$'s, $e_2$
$\vep$'s, $\ldots, d_r$ $\delta$'s, $e_r$ $\vep$'s and sequences
of $\pm 1$'s: $(\xi_i)_{1\le i \le n} \cup (\eta_j)_{1\le j \le m}$
(all the $d_i$ and $e_j$ are positive except possibly $d_1 =0$ or
$e_r =0$). Note that a Dynkin diagram contains a short {\em odd}
root exactly when $e_r =0$.
Let
$$
\texttt d_u =\sum_{a=1}^u d_a, \quad
\texttt e_u
=\sum_{a=1}^u e_a
$$ for $u=1, \ldots, r$, and let $\texttt d_0
=\texttt e_0 =0$. Note $\texttt d_r=n, \texttt e_r=m$. More
precisely, there exist a permutation $s$ of $\{1, \ldots, n\}$
and a permutation $t$ of $\{1, \ldots, m\}$, so that the simple
roots for $\mc B$ are given by
\begin{align*} \xi_i \delta_{s(i)} -\xi_{i+1} \delta_{s(i+1)},&
\quad 1\le i \le n,\; i \not \in \{\texttt d_u | u=1,\ldots, r\};
  \\
\eta_j \vep_{t(j)} -\eta_{j+1} \vep_{t(j+1)},& \quad 1 \le j \le
m, \; j \not \in \{\texttt e_u | u=1,\ldots, r\};
  \\
\xi_{\texttt d_u} \delta_{s({\texttt d_u})} - \eta_{1+\texttt
e_{u-1}} \vep_{t({1+\texttt e_{u-1}})}, &\quad \text{for } 1\le u
\le r \text{ if } e_r>0\;\; (\text{or } 1\le u< r \text{ if }
e_r=0);
  \\
\eta_{\texttt e_u} \vep_{t({\texttt e_u})} -\xi_{1+\texttt d_u}
\delta_{s({1+\texttt d_u})}, & \quad u=1,\ldots, r-1;
  \\
\eta_{\texttt e_r} \vep_{t({\texttt e_r})}, &\quad \text{ if } e_r>0  \quad
\quad (\text{or }  \xi_{\texttt d_r} \delta_{s({\texttt d_r})}
\text{ if } e_r=0).
\end{align*}

Recall a partition $\la=(\la_1,\la_2,\ldots)$ is called an $(n|m)$-{\em hook partition},
if $\la_{n+1}\le m$ (cf. \cite{BR, S}). For such a $\la$, we define
$$
\la^\# =(\la_1,\ldots, \la_n, \nu_1, \ldots, \nu_m),
$$
where $(\nu_1, \ldots, \nu_m)$ is the conjugated partition
of $(\la_{n+1}, \la_{n+2}, \ldots)$.

\begin{lem} \cite{K2} \label{kac hwt}
The irreducible $\osp(2m+1|2n)$-module of integral highest weight
$\sum_{i=1}^n\la_i\delta_i + \sum_{j=1}^m\ov{\la}_j\vep_j$ with respect to the standard
Borel subalgebra is finite dimensional if and only if
$(\la_1,\ldots,\la_n,\ov{\la}_1,\ldots,\ov{\la}_m)=\la^\#$ for some $(n|m)$-hook
partition $\la$.
\end{lem}
We denote by $L'(\osp(2m+1|2n),\la^\#)$ these irreducible $\osp(2m+1|2n)$-modules with
respect to the standard Borel subalgebra, to distinguish from earlier notation used for
irreducible modules with respect to different Borel subalgebra. Actually the finite
dimensionality criterion was given in \cite{K2} in terms of Dynkin labels, which is known
to be equivalent to the more natural labeling above in terms of $(n|m)$-hook partitions
(cf.  \cite{SW}). Same remark applies to the finite dimensionality criterion for
$\osp(2m|2n)$ in \lemref{hwt:standard:spo} below.

\begin{example}  \label{diag 9|10}
Suppose that the corresponding Dynkin diagram of a Borel
subalgebra of $\osp(9|10)$ is as follows:
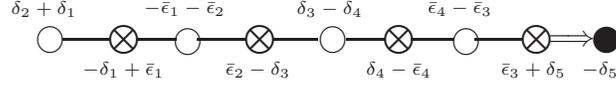
\begin{center}
\hskip -3cm \setlength{\unitlength}{0.16in}
\begin{picture}(24,4)
\put(6,2){\makebox(0,0)[c]{$\bigcirc$}}
\put(8.4,2){\makebox(0,0)[c]{$\bigotimes$}}
\put(10.5,1.95){\makebox(0,0)[c]{$\bigcirc$}}
\put(12.85,2){\makebox(0,0)[c]{$\bigotimes$}}
\put(15.25,2){\makebox(0,0)[c]{$\bigcirc$}}
\put(17.4,2){\makebox(0,0)[c]{$\bigotimes$}}
\put(19.6,1.95){\makebox(0,0)[c]{$\bigcirc$}}
\put(21.9,2){\makebox(0,0)[c]{$\bigotimes$}}
\put(24.2,2){\circle*{0.9}}
%
\put(6.4,2){\line(1,0){1.55}} \put(8.82,2){\line(1,0){1.3}}
\put(10.8,2){\line(1,0){1.6}} \put(13.28,2){\line(1,0){1.45}}
\put(15.7,2){\line(1,0){1.25}} \put(17.8,2){\line(1,0){1.4}}
\put(19.9,2){\line(1,0){1.5}} \put(22.3,1.8){$\Longrightarrow$}
\put(5.8,3){\makebox(0,0)[c]{\tiny $\delta_2+\delta_1$}}
\put(8.4,1){\makebox(0,0)[c]{\tiny $-\delta_1+\vep_1$}}
\put(10.4,3){\makebox(0,0)[c]{\tiny $-\vep_1-\vep_2$}}
\put(12.8,1){\makebox(0,0)[c]{\tiny $\vep_2-\delta_3$}}
\put(15.15,3){\makebox(0,0)[c]{\tiny $\delta_3-\delta_4$}}
\put(17.4,1){\makebox(0,0)[c]{\tiny $\delta_4-\vep_4$}}
\put(19.4,3){\makebox(0,0)[c]{\tiny $\vep_4-\vep_3$}}
\put(21.8,1){\makebox(0,0)[c]{\tiny $\vep_3+\delta_5$}}
\put(24,1){\makebox(0,0)[c]{\tiny $-\delta_5$}}
\end{picture}
\end{center}
We read off from the above a signed sequence
with indices
$\delta_2(-\delta_1)(-\vep_1)\vep_2\delta_3\delta_4\vep_4\vep_3(-\delta_5)$.
In particular, we obtain a sequence
$\delta\delta\vep\vep\delta\delta\vep\vep\delta$ by ignoring the
signs and indices. In this case, $d_1=d_2=2,d_3=1$, and
$e_1=e_2=2$.  Furthermore, the sequences $(\xi_i)_{1\le i \le 5}$
and $(\eta_j)_{1\le j\le 4}$ are $(1,-1,1,1,-1)$ and $(-1,1,1,1)$,
respectively.
\end{example}

Define the {\em block Frobenius coordinates} $(p_i|q_j)$ of an
$(n|m)$-hook partition $\la$ associated to $\mc B$ as follows. For
$1\le i \le n, 1\le j \le m$, let
\begin{eqnarray*}
p_i &= \max \{\la_i - \texttt e_u, 0\}, &\text{ if }\texttt d_u < i
\leq \texttt d_{u+1}  \text{ for some } 0\le u\le r-1
  \\
q_j  &= \max \{\la_j' - \texttt d_{u+1}, 0\},  & \text{ if }   \texttt
e_u+1 < j \leq \texttt e_{u+1}  \text{ for some } 0\le u\le r-1.
\end{eqnarray*}
It is elementary to read off the block Frobenius coordinates of
$\la$ from the Young diagram of $\la$ in general, as illustrated
by the next example.

\begin{example}\label{frob:block}
Consider the $(5,4)$-hook diagram $\la =(14,11,8,8,7,4,3,2).$ The
block Frobenius coordinates associated with $\mc B$ from Example
\ref{diag 9|10} for $\la$ is:
$$p_1=14, p_2=11, p_3=p_4=6, p_5=3;\quad q_1=q_2=6, q_3=3, q_4=2.
$$
These are read off from the Young diagram of $\la$ by following
the $\vep\delta$ sequence $\delta\delta\vep\vep\delta\delta\vep\vep\delta$ as follows:
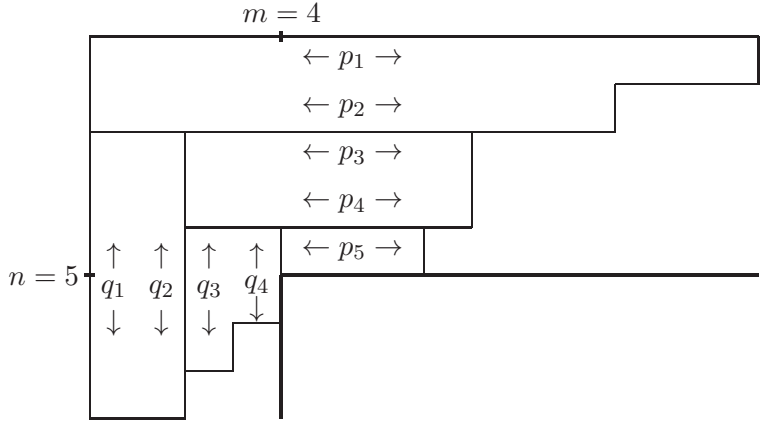
\begin{center}
\hskip 1cm \setlength{\unitlength}{0.25in}
\begin{picture}(15,9)
\put(0,6){\line(0,1){2}}
\put(0,8){\line(1,0){14}}
\put(14,8){\line(0,-1){1}}
\put(14,7){\line(-1,0){3}}
\put(11,7){\line(0,-1){1}}
\put(11,6){\line(-1,0){11}}
\put(0,6){\line(0,-1){6}}
\put(0,0){\line(1,0){2}}
\put(2,0){\line(0,1){6}}
\put(2,4){\line(1,0){6}}
\put(8,4){\line(0,1){2}}
\put(2,1){\line(1,0){1}}
\put(3,1){\line(0,1){1}}
\put(3,2){\line(1,0){1}}
\put(4,2){\line(0,1){2}}
\put(4,3){\line(1,0){3}} \put(7,3){\line(0,1){1}}
\put(5.5,7.5){\makebox(0,0)[c]{$\leftarrow p_1\rightarrow$}}
\put(5.5,6.5){\makebox(0,0)[c]{$\leftarrow p_2\rightarrow$}}
\put(5.5,5.5){\makebox(0,0)[c]{$\leftarrow p_3\rightarrow$}}
\put(5.5,4.5){\makebox(0,0)[c]{$\leftarrow p_4\rightarrow$}}
\put(5.5,3.5){\makebox(0,0)[c]{$\leftarrow p_5\rightarrow$}}
\put(0.5,3.4){\makebox(0,0)[c]{$\uparrow$}}
\put(0.5,2){\makebox(0,0)[c]{$\downarrow$}}
\put(0.5,2.7){\makebox(0,0)[c]{$q_1$}}
\put(1.5,3.4){\makebox(0,0)[c]{$\uparrow$}}
\put(1.5,2){\makebox(0,0)[c]{$\downarrow$}}
\put(2.5,3.4){\makebox(0,0)[c]{$\uparrow$}}
\put(2.5,2){\makebox(0,0)[c]{$\downarrow$}}
\put(3.5,3.4){\makebox(0,0)[c]{$\uparrow$}}
\put(3.5,2.3){\makebox(0,0)[c]{$\downarrow$}}
\put(1.5,2.7){\makebox(0,0)[c]{$q_2$}}
\put(2.5,2.7){\makebox(0,0)[c]{$q_3$}}
\put(3.5,2.8){\makebox(0,0)[c]{$q_4$}}
\put(4,3){\linethickness{1pt}\line(0,-1){3}}
\put(4,3){\linethickness{1pt}\line(1,0){10}}
\put(-0.1,3){\linethickness{1pt}\line(1,0){0.2}}
\put(4,7.9){\linethickness{1pt}\line(0,1){0.2}}
\put(-1.7,2.8){$n=5$} \put(3.2,8.3){$m=4$}
\end{picture}
\vskip 0.5cm
\end{center}
\end{example}

\begin{thm}    \label{hwt change}
Let $\la$ be an $(n|m)$-hook partition. Let $\mc B$ be a Borel
subalgebra of $\osp(2m+1|2n)$ and retain the above notation. Then,
the $\mc B$-highest weight of the simple $\osp(2m+1|2n)$-module
$L'(\osp(2m+1|2n),\la^\#)$ is $$\la^{\mc B} :=\sum_{i=1}^n \xi_i p_i
\delta_{s(i)} +\sum_{j=1}^m \eta_j q_j \vep_{t(j)}.$$
\end{thm}

\begin{proof}
Let us consider an odd reflection that changes a Borel subalgebra
$\mc B_1$ to $\mc B_2$. Assume the theorem holds for $\mc B_1$. We
observe by Lemma~\ref{hwt odd} that the statement of the theorem for
$\mc B_2$ follows from the validity of the theorem for $\mc B_1$.
The statement of the theorem is apparently consistent with a
change of Borel subalgebras induced from a real reflection, and
all Borel subalgebras are linked by a sequence of real and odd
reflections. Hence, once we know the theorem holds for one
particular Borel subalgebra, it holds for all. We finally note
that the theorem holds for the standard Borel subalgebra $\mc
B^{\text{st}}$, which corresponds to the sequence of $n$
$\delta$'s followed by $m$ $\vep$'s with all signs $\xi_i$ and
$\eta_j$ being positive, i.e., $\la^{\mc B^{\text{st}}} =\la^\#$.
\end{proof}

\begin{example}
With respect to the Borel $\mc B$ of $\osp(9|10)$ as in
Example~\ref{diag 9|10}, the $\mc B$-extremal weight of $L'(\osp(9|10),\la^\#)$
for $\la$ as in Example~\ref{frob:block} equals to
\begin{align*}
-11\delta_1+14\delta_2+6\delta_3+6\delta_4-3\delta_5
-6\vep_1+6\vep_2+2\vep_3+3\vep_4.
\end{align*}
\end{example}

\begin{cor}
Every finite dimensional irreducible $\osp(2m+1|2n)$-module of integral highest weight is
self-contragradient.
\end{cor}
\begin{proof}
Denote by $\mc B^{\text{op}}$ the opposite Borel to the standard
one $\mc B^{\text{st}}$. It follows by Theorem~\ref{hwt change}
that the $\mc B^{\text{op}}$-extremal weight of the module
$L'(\osp(2m+1|2n) ,\la^\#)$ is $-\la^\#$.
\end{proof}

Recall that the following Dynkin diagram of $\osp(2m+1|2n)$ and of its
(trivial) central extension $\SG^\mf{b}_n$ has been in use from
the point of view of super duality and it is opposite to the one
associated to the standard Borel $\mc B^{\text{st}}$.
 \begin{center}
\hskip -3cm \setlength{\unitlength}{0.16in}
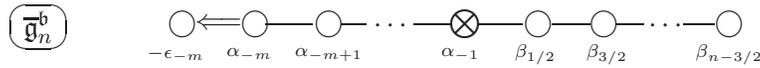
\begin{picture}(24,4)
\put(5.6,2){\makebox(0,0)[c]{$\bigcirc$}}
\put(8,2){\makebox(0,0)[c]{$\bigcirc$}}
\put(10.4,2){\makebox(0,0)[c]{$\bigcirc$}}
\put(14.85,2){\makebox(0,0)[c]{$\bigotimes$}}
\put(17.25,2){\makebox(0,0)[c]{$\bigcirc$}}
\put(19.4,2){\makebox(0,0)[c]{$\bigcirc$}}
\put(23.5,2){\makebox(0,0)[c]{$\bigcirc$}}
\put(8.35,2){\line(1,0){1.5}} \put(10.82,2){\line(1,0){0.8}}
\put(13.2,2){\line(1,0){1.2}} \put(15.28,2){\line(1,0){1.45}}
\put(17.7,2){\line(1,0){1.25}} \put(19.81,2){\line(1,0){0.9}}
\put(22,2){\line(1,0){1}}
\put(6.8,2){\makebox(0,0)[c]{$\Longleftarrow$}}
\put(12.5,1.95){\makebox(0,0)[c]{$\cdots$}}
\put(21.5,1.95){\makebox(0,0)[c]{$\cdots$}}
\put(5.4,1){\makebox(0,0)[c]{\tiny $-\epsilon_{-m}$}}
\put(7.8,1){\makebox(0,0)[c]{\tiny $\alpha_{-m}$}}
\put(10.4,1){\makebox(0,0)[c]{\tiny $\alpha_{-m+1}$}}
\put(14.7,1){\makebox(0,0)[c]{\tiny $\alpha_{-1}$}}
\put(17.15,1){\makebox(0,0)[c]{\tiny $\beta_{1/2}$}}
\put(19.5,1){\makebox(0,0)[c]{\tiny $\beta_{3/2}$}}
\put(23.5,1){\makebox(0,0)[c]{\tiny $\beta_{n-3/2}$}}
\put(0,1.2){{\ovalBox(1.8,1.4){${\SG}^\mf{b}_n$}}}
\end{picture}
\end{center}
Setting $\vep_{j} =\ep_{-m+j-1}$ and $\delta_i
=\ep_{n-i+1/2}$ to match the notation in this section with the one
used earlier, we have the following immediate corollary of
\lemref{kac hwt} and \thmref{hwt change}.

\begin{cor}\label{aux:finite1}
An irreducible integral highest weight $\osp(2m+1|2n)$-module with respect to the Borel
subalgebra corresponding to
\makebox(22,0){$\oval(22,15)$}\makebox(-22,8){${\SG}_n^\mf{b}$} is finite dimensional if
and only if the highest weight is of the form
\begin{equation}\label{B-type:finite:hw2}
-\sum_{j=1}^{m}\max\{\la'_{j}-n, 0\} \, \epsilon_{-j}
-\sum_{i=1}^n\la_{n-i+1}\epsilon_{i-1/2},
\end{equation}
where $\la=(\la_1,\la_2,\ldots)$ is an $(n|m)$-hook partition.
\end{cor}

\subsection{Extremal weights for $\osp(2m|2n)$}

Let us denote the weights of the natural $\osp(2m|2n)$-module
$\C^{2n|2m}$ by $\pm \delta_i, \pm \vep_j$ for $1\le i\le n, 1\le
j \le m$. The {\em standard} Borel subalgebra $\mc B^{\text{st}}$
of $\osp(2m|2n)$ is the one associated to the following set of
simple roots
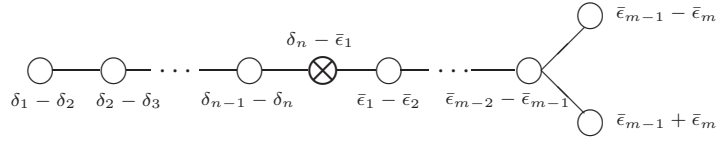
\begin{center}
\hskip -3cm \setlength{\unitlength}{0.16in}
\begin{picture}(24,4.5)
\put(6,2){\makebox(0,0)[c]{$\bigcirc$}}
\put(8.4,2){\makebox(0,0)[c]{$\bigcirc$}}
\put(12.85,2){\makebox(0,0)[c]{$\bigcirc$}}
\put(15.25,2){\makebox(0,0)[c]{$\bigotimes$}}
\put(17.4,2){\makebox(0,0)[c]{$\bigcirc$}}
\put(22,2){\makebox(0,0)[c]{$\bigcirc$}}
\put(24,3.8){\makebox(0,0)[c]{$\bigcirc$}}
\put(24,.3){\makebox(0,0)[c]{$\bigcirc$}}
\put(6.4,2){\line(1,0){1.55}} \put(8.82,2){\line(1,0){0.8}}
\put(11.2,2){\line(1,0){1.2}} \put(13.28,2){\line(1,0){1.45}}
\put(15.7,2){\line(1,0){1.25}} \put(17.8,2){\line(1,0){0.9}}
\put(20.1,2){\line(1,0){1.4}} \put(22.4,2){\line(1,1){1.4}}
\put(22.4,2){\line(1,-1){1.4}}
\put(10.5,1.95){\makebox(0,0)[c]{$\cdots$}}
\put(19.6,1.95){\makebox(0,0)[c]{$\cdots$}}
\put(6.1,1){\makebox(0,0)[c]{\tiny $\delta_1-\delta_2$}}
\put(8.9,1){\makebox(0,0)[c]{\tiny $\delta_2-\delta_3$}}
\put(12.8,1){\makebox(0,0)[c]{\tiny $\delta_{n-1}-\delta_{n}$}}
\put(15.15,3){\makebox(0,0)[c]{\tiny $\delta_n-\vep_1$}}
\put(17.4,1){\makebox(0,0)[c]{\tiny $\vep_1-\vep_2$}}
\put(21.3,1){\makebox(0,0)[c]{\tiny $\vep_{m-2}-\vep_{m-1}$}}
\put(26.5,3.8){\makebox(0,0)[c]{\tiny $\vep_{m-1}-\vep_{m}$}}
\put(26.5,0.3){\makebox(0,0)[c]{\tiny $\vep_{m-1}+\vep_{m}$}}
\end{picture}
\end{center}

There are two kinds of Dynkin diagrams and corresponding Borel subalgebras
for $\osp(2m|2n)$:
\begin{enumerate}
\item[(i)] Diagrams of $|$-shape, i.e., Dynkin diagrams with a
long simple root $\pm 2\delta_i$.

\item[(ii)] Diagrams of {\Large $\Ydown$}-shape, i.e., Dynkin
diagrams with no long simple root.
\end{enumerate}

We will follow the notation for $\osp(2m+1|2n)$ in
Subsection~\ref{sec:type B} for sets of simple roots in terms of
signed $\vep\delta$ sequences, so we have permutations $s, t$, and
signs $\xi_i, \eta_j$. We fix an ambiguity on the choice of the
sign $\eta_m$ associated to a Borel $\mc B$ of {\Large
$\Ydown$}-shape, by demanding the total number of negative signs
among $\eta_j (1\le j \le m)$ to be always even.

Let $\la$ be an $(n|m)$-hook partition, and let the block
Frobenius coordinates $(p_i|q_j)$ be as defined in
Subsection~\ref{sec:type B}. Introduce the following weights:
\begin{eqnarray*}
\la^{\mc B} &:=& \sum_{i=1}^n \xi_i p_i \delta_{s(i)}
+\sum_{j=1}^m \eta_j q_j \vep_{t(j)},
 \\
\la^{\mc B}_- &:=& \sum_{i=1}^n \xi_i p_i \delta_{s(i)}
+\sum_{j=1}^{m-1} \eta_j q_j \vep_{t(j)} -\eta_m q_m \vep_{t(m)}.
\end{eqnarray*}
The weight $\la^{\mc B}_-$ will only be used for Borel $\mc B$ of
{\Large $\Ydown$}-shape. Note that $\la^{\mc B^{\text{st}}}
=\la^\#$ and we shall denote $\la^{\#}_- :=\la^{\mc
B^{\text{st}}}_-$.

Given a Borel $\mc B$ of $|$-shape, we define $\text{s}(\mc B)$ to
be the sign of $\prod_{j=1}^m \eta_j$.

\begin{lem}\label{hwt:standard:spo} \cite{K2}
The irreducible $\osp(2m|2n)$-module of integral highest weight of the form
$\sum_{i=1}^n\mu_i\delta_i + \sum_{j=1}^m\ov{\mu}_j\vep_j$ with respect to the standard
Borel subalgebra is finite dimensional if and only if
$(\mu_1,\ldots,\mu_n,\ov{\mu}_1,\ldots,\ov{\mu}_m)$ is either $\la^\#$ or $\la^{\#}_-$
for some $(n|m)$-hook partition $\la$.
\end{lem}

We shall denote these irreducible $\osp(2m|2n)$-modules
with respect to the standard Borel by
$L'(\osp(2m|2n),\la^\#)$ and
$L'(\osp(2m|2n),\la^\#_-)$. By a similar argument as for
Theorem~\ref{hwt change}, we establish the following.
\begin{thm}\label{hwt change 2}
Let $\la$ be an $(n|m)$-hook partition.
\begin{enumerate}
\item Assume $\mc B$ is of {\Large $\Ydown$}-shape. Then,
\begin{itemize}
\item[(i)] $\la^{\mc B}$ is the $\mc B$-extremal weight for the
module $L'(\osp(2m|2n), \la^\#)$.

\item[(ii)] $\la^{\mc B}_-$ is of the $\mc B$-extremal weight for the
module $L'(\osp(2m|2n), \la^{\#}_-)$.
\end{itemize}

\item Assume $\mc B$ is of $|$-shape. Then,
\begin{itemize}
\item[(i)] $\la^{\mc B}$ is the $\mc B$-extremal weight for
$L'(\osp(2m|2n), \la^\#)$ if $\text{s}(\mc B) =+$.

\item[(ii)] $\la^{\mc B}$ is the $\mc B$-extremal weight for
$L'(\osp(2m|2n), \la^{\#}_-)$ if $\text{s}(\mc B) =-$.
\end{itemize}
\end{enumerate}
\end{thm}

\begin{cor}
For $m$ even, every finite dimensional irreducible $\osp(2m|2n)$-module of integral
highest weight is self-contragradient.
\end{cor}
\begin{rem}
The remaining $\mc B$-extremal weights for the modules $L'(\osp(2m|2n),
\la^\#)$ when $\text{s}(\mc B) =-$ or for the modules
$L'(\osp(2m|2n), \la^\#_-)$ when $\text{s}(\mc B) =+$ are
rather complicated and do not seem to afford a uniform simple
answer.
\end{rem}

The following Dynkin diagram of $\osp(2m|2n)$ or
$\SG_n^\dd$ that has been in use for super duality
is opposite to the standard Borel $\mc B^{\text{st}}$.
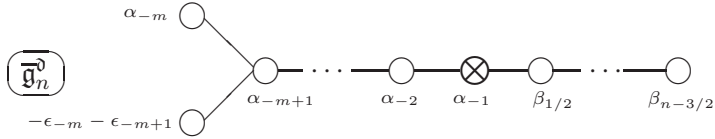
\begin{center}
\hskip -3cm \setlength{\unitlength}{0.16in}
\begin{picture}(24,5)
\put(6,0.3){\makebox(0,0)[c]{$\bigcirc$}}
\put(6,3.8){\makebox(0,0)[c]{$\bigcirc$}}
\put(8.4,2){\makebox(0,0)[c]{$\bigcirc$}}
\put(12.85,2){\makebox(0,0)[c]{$\bigcirc$}}
\put(15.25,2){\makebox(0,0)[c]{$\bigotimes$}}
\put(17.4,2){\makebox(0,0)[c]{$\bigcirc$}}
\put(21.9,2){\makebox(0,0)[c]{$\bigcirc$}}
\put(6.4,0.4){\line(1,1){1.6}}
\put(6.4,3.7){\line(1,-1){1.6}}
\put(8.82,2){\line(1,0){0.8}}
\put(11.2,2){\line(1,0){1.2}} \put(13.28,2){\line(1,0){1.45}}
\put(15.7,2){\line(1,0){1.25}} \put(17.8,2){\line(1,0){0.9}}
\put(20.1,2){\line(1,0){1.4}}
%
\put(10.5,1.95){\makebox(0,0)[c]{$\cdots$}}
\put(19.6,1.95){\makebox(0,0)[c]{$\cdots$}}
%
\put(3,0.3){\makebox(0,0)[c]{\tiny $-\epsilon_{-m}-\epsilon_{-m+1}$}}
\put(4.5,3.8){\makebox(0,0)[c]{\tiny $\alpha_{-m}$}}
\put(8.9,1){\makebox(0,0)[c]{\tiny $\alpha_{-m+1}$}}
\put(12.8,1){\makebox(0,0)[c]{\tiny $\alpha_{-2}$}}
\put(15.15,1){\makebox(0,0)[c]{\tiny $\alpha_{-1}$}}
\put(17.8,1){\makebox(0,0)[c]{\tiny $\beta_{1/2}$}}
\put(22,1){\makebox(0,0)[c]{\tiny $\beta_{n-3/2}$}}
\put(0,1.2){{\ovalBox(1.8,1.4){$\SG_n^\dd$}}}
\end{picture}
\end{center}

Setting
$\vep_{j} =\ep_{-m+j-1}$ and $\delta_i =\ep_{n-i+1/2}$ to match notations, we record the following
corollary of \lemref{hwt:standard:spo} and \thmref{hwt change 2}.

\begin{cor}\label{aux:finite2}
An irreducible integral highest weight $\osp(2m|2n)$-module with respect to the Borel
subalgebra corresponding to \makebox(24,0){$\oval(24,14)$}\makebox(-24,8){${\SG}_n^\dd$}
is finite dimensional if and only if the highest weight is of the form
\begin{equation}\label{B-type:finite:hw1}
\pm\max \{\la'_{m}-n, 0\} \,\epsilon_{-m}
-\sum_{j=1}^{m-1}\max\{\la'_{j}-n, 0\} \,\epsilon_{-j}
-\sum_{i=1}^n\la_{n-i+1}\epsilon_{i-1/2},
\end{equation}
where $\la=(\la_1,\la_2,\ldots)$ is an $(n|m)$-hook partition.
\end{cor}

\begin{rem}\label{finite:dim:character}
From Corollaries \ref{aux:finite1} and \ref{aux:finite2} it follows that, after passing
to the central extension $\SG_n$ on which the center $K$ acts as a scalar multiplication
by $d\in\Z$, the weights in \eqnref{B-type:finite:hw2} and \eqnref{B-type:finite:hw1}
lie in $\bar{P}^+_n$ whenever $d\le -\la_1$. Hence, \thmref{character} and
\lemref{lem:trunc} provide a complete solution to the {\em finite dimensional}
irreducible character problem for the ortho-symplectic Lie superalgebras.
\end{rem}

\begin{rem}
Recall \cite{BR,S} that finite dimensional irreducible {\em polynomial}
$\gl(n|m)$-modules are exactly the highest weight modules $L'(\gl(n|m), \la^\#)$  with
respect to the standard Borel subalgebra parametrized by $(n|m)$-hook partitions $\la$.
One can assign to any Borel subalgebra $\mc{B}$ of $\gl(n|m)$ an $\vep\delta$ sequence as
in \secref{sec:type B}, but now with $\xi_i=\eta_j=1$, for all $i,j$. By the same
argument as for Theorem~\ref{hwt change},  we can show that the highest weights of the
polynomial representations of $\gl(n|m)$ with respect to $\mc{B}$ is given by $\la^{\mc
B} =\sum_{i=1}^n p_i \delta_{s(i)} +\sum_{j=1}^m q_j \vep_{t(j)}. $
\end{rem}
\subsection{}
\label{opposite:borel}
By flipping from the left to right the Dynkin diagram of
\makebox(23,0){$\oval(20,12)$}\makebox(-20,8){$\mf{k}^\xx$} and changing all
the simple roots therein to their opposites, we obtain a Dynkin diagram\
\makebox(23,0){$\oval(20,12)$}\makebox(-20,8){$^\texttt{o}\mf{k}^\xx$}
corresponding to the opposite Borel subalgebras, where
$\mf{x=b,b^\bullet,c,d}$. Similarly, by flipping
the Dynkin diagrams
\makebox(23,0){$\oval(20,12)$}\makebox(-20,8){$\mf{T}_n$},
\makebox(23,0){$\oval(20,14)$}\makebox(-20,8){$\ov{\mf{T}}_n$} and
\makebox(23,0){$\oval(20,15)$}\makebox(-20,8){$\wt{\mf{T}}_n$} and changing
all signs of the simple roots for $n\in\N\cup\{\infty\}$, we obtain the Dynkin diagrams
\makebox(23,0){$\oval(20,12)$}\makebox(-20,8){$^\texttt{o}\mf{T}_n$},
\makebox(23,0){$\oval(20,14)$}\makebox(-20,8){$^\texttt{o}\ov{\mf{T}}_n$}
and
\makebox(23,0){$\oval(20,15)$}\makebox(-20,8){$^\texttt{o}\wt{\mf{T}}_n$},
respectively, of the opposite Borel subalgebras.
We form the diagrams corresponding to the Borel
subalgebras opposite to \eqnref{Dynkin:combined} as follows:
\begin{equation*}\label{Dynkin:combined:op}
\hskip -3cm \setlength{\unitlength}{0.16in}
\begin{picture}(24,1)
\put(5.0,0.5){\makebox(0,0)[c]{{\ovalBox(1.6,1.2){$^\texttt{o}\mf{T}_n$}}}}
\put(5.8,0.5){\line(1,0){1.85}}
\put(8.5,0.5){\makebox(0,0)[c]{{\ovalBox(1.6,1.2){$^\texttt{o}\mf{k}^\xx$}}}}
\put(15,0.5){\makebox(0,0)[c]{{\ovalBox(1.6,1.2){$^\texttt{o}\ov{\mf{T}}_n$}}}}
\put(15.8,0.5){\line(1,0){1.85}}
\put(18.5,0.5){\makebox(0,0)[c]{{\ovalBox(1.6,1.2){$^\texttt{o}\mf{k}^\xx$}}}}
\put(25,0.5){\makebox(0,0)[c]{{\ovalBox(1.6,1.2){$^\texttt{o}\wt{\mf{T}}_n$}}}}
\put(25.8,0.5){\line(1,0){1.85}}
\put(28.5,0.5){\makebox(0,0)[c]{{\ovalBox(1.6,1.2){$^\texttt{o}\mf{k}^\xx$}}}}
\end{picture}
\end{equation*}
The corresponding Lie superalgebras are again $\G$, $\SG$ and
$\DG$, respectively.

The arguments in Sections \ref{sec:O} and \ref{sec:character} can
be adapted easily to allow us to compare correspondingly defined
parabolic categories $^\texttt{o}\mc{O}$, $^\texttt{o}\ov{\mc{O}}$
and $^\texttt{o}\wt{\mc{O}}$ using these opposite Borel
subalgebras, whose precise definitions are evident. We note that
for the corresponding set of weights $^\texttt{o}P^+$ of the form
\begin{align*}
\sum_{i=1}^{m}\la_{i}\epsilon_{-i} -
\sum_{j\in\N}\la^+_{j}\epsilon_{j} + d\La_0,\quad d\in\C,
\end{align*}
to satisfy the corresponding {dominant condition} we require, besides the obvious
dominant condition on the standard Levi subalgebra of
\makebox(23,0){$\oval(20,12)$}\makebox(-20,8){$^\texttt{o}\mf{k}^\xx$}, also that
$\la^+=(\la^+_1,\la^+_2,\ldots)$ is a partition. This allows us to prove an analogous
version of \thmref{character} and thus to compute irreducible characters of Lie
superalgebras in terms of irreducible characters of Lie algebras. Also the results in
Section~ \ref{sec:homology} and section~ \ref{sec:category} have fairly straightforward
analogues in $^\texttt{o}\mc{O}$, $^\texttt{o}\ov{\mc{O}}$ and $^\texttt{o}\wt{\mc{O}}$
as well.  In particular, we can prove equivalences of the corresponding finitely
generated module subcategories following the strategy of \secref{sec:category}.

Besides of its own interest, another virtue of this opposite version of super duality
lies in the ease of calculation of finite dimensional irreducible characters of modules
over the finite dimensional ortho-symplectic Lie superalgebras.  As the highest weight
modules over $\SG$ in this setup already have highest weights over the standard Borel
subalgebras, the knowledge of extremal weights for finite dimensional irreducible modules
is no longer needed to imply that solution of the irreducible character problem in the
category $^\texttt{o}\ov{\mc{O}}$ and $^\texttt{o}\ov{\mc{O}}_n$ also solves the finite
dimensional irreducible character problem.

\end{document}